\documentclass[10pt, a4paper]{amsart}
\usepackage{amssymb,graphicx}
\usepackage{mathrsfs, bbm, slashed}
\usepackage{color}
\usepackage{enumerate}
\usepackage{a4wide}

\newcommand{\R}{\mathbb{R}}
\newcommand{\C}{\mathbb{C}}
\newcommand{\Z}{\mathbb{Z}}
\newcommand{\N}{\mathbb{N}}
\newcommand{\T}{\mathbb{T}}
\newcommand{\bS}{\mathbb{S}}
\newcommand{\bB}{\mathbb{B}}

\newcommand{\cO}{\mathcal{O}}

\newcommand{\sD}{\mathscr{D}}
\newcommand{\sE}{\mathscr{E}}
\newcommand{\sI}{\mathscr{I}}
\newcommand{\sJ}{J}
\newcommand{\sL}{\mathscr{L}}
\newcommand{\sH}{\mathscr{H}}
\newcommand{\sK}{\mathscr{K}}
\newcommand{\sM}{\mathscr{M}}
\newcommand{\sR}{\mathscr{R}}

\newcommand{\sU}{\mathscr{U}}
\newcommand{\sV}{\mathscr{V}}
\newcommand{\sW}{\mathscr{W}}

\newcommand{\co}{\mathfrak{c}_{0}}

 \newcommand{\psidos}{$\Psi$DOs}

\def\Xint#1{\mathchoice
{\XXint\displaystyle\textstyle{#1}}
{\XXint\textstyle\scriptstyle{#1}}
{\XXint\scriptstyle\scriptscriptstyle{#1}}
{\XXint\scriptscriptstyle\scriptscriptstyle{#1}}
\!\int}
\def\XXint#1#2#3{{\setbox0=\hbox{$#1{#2#3}{\int}$}
\vcenter{\hbox{$#2#3$}}\kern-.5\wd0}}
\def\dashint{\Xint-}

\newcommand{\bint}{\ensuremath{\dashint}}
\newcommand{\gbint}{\sideset{^g\!\!}{}\bint}
\newcommand{\gubint}{\sideset{^{g_1}\!\!}{}\bint}

\newcommand{\op}{\operatorname}

\newcommand{\Sp}{\op{Sp}}
\newcommand{\Vol}{\op{Vol}}
\newcommand{\Tr}{\op{Tr}}
\newcommand{\Trw}{\Tr_{\omega}}
\newcommand{\dist}{\op{dist}}
\newcommand{\ran}{\op{ran}}
\newcommand{\Span}{\op{Span}}

\newcommand{\Com}{\op{Com}}

\newcommand{\rk}{\op{rk}}
\newcommand{\car}{\mathbbm{1}}

\newcommand{\SU}{\op{SU}}
\newcommand{\su}{\op{su}}

\newcommand{\scal}[2]{\ensuremath{\left\langle #1 | #2 \right\rangle}}
\newcommand{\ketbra}[2]{\ensuremath{|#1\left\rangle\! \right\langle#2|}}

\newcommand{\dom}{\op{dom}}

\numberwithin{equation}{section}

\newtheorem{theorem}{Theorem}[section]
\newtheorem{proposition}[theorem]{Proposition}
\newtheorem{corollary}[theorem]{Corollary}

\newtheorem{theoremalph}{Theorem}
\newtheorem{lemmalph}[theoremalph]{Lemma}
\newtheorem{propositionalph}[theoremalph]{Proposition}

\newtheorem{lemma}[theorem]{Lemma}

\newtheorem{conjecture*}{Conjecture}

\theoremstyle{definition}
\newtheorem{definition}[theorem]{Definition}

\theoremstyle{remark}
\newtheorem{example}[theorem]{Example}
\newtheorem{remark}[theorem]{Remark}

\newtheorem{claim*}{Claim}

\newcommand{\RV}{\textup{RV}}
\newcommand{\shD}{\slashed{D}}
\newcommand{\shS}{\slashed{S}}

\newcommand{\EL}{\op{EL}}

\newcommand{\loc}{\textup{loc}}

\title{Nonclassical Weyl Laws and\\ Connes' Integration for Weak Lorentz Ideals, I}
\date{\today}

\author{Rapha\"el Ponge}
 \address{Department of Mathematics and Statistics, University of Ottawa, Canada}
  \curraddr{Department of Mathematics and Statistics, University of Wyoming, Laramie, WY, USA}
\email{ponge.math@icloud.com}

\author{Yongqiang Tian}
 \address{School of Mathematics and Statistics, Central South University, Changsha, China}
 \email{tianyongqiang@csu.edu.cn}



\begin{document}
\begin{abstract}
Motivated by nonclassical Weyl laws arising in various contexts (including Connes' approach to the Riemann Hypothesis), we develop a systematic theory of Dixmier traces and Connes' noncommutative integration for weak Lorentz ideals associated with regularly varying functions. A key ingredient is an asymptotic additivity property for eigenvalue partial sums, obtained by combining Karamata's theorem with results of Kalton and Lord-Sukochev-Zanin. This yields a direct construction of Dixmier traces in terms of eigenvalue sequences and a complete spectral characterization of measurable operators, answering a question of Connes in this general setting. We also extend to weak Lorentz ideals the Birman-Solomyak perturbation theory for eigenvalue and singular-value asymptotics. Weyl operators (those admitting precise asymptotic limits for their rescaled eigenvalue sequences) are shown to form a closed subset of the ideal, stable under compact perturbations, extending classical results of Weyl and Birman-Solomyak. We further study \emph{strong measurability} (measurability with respect to all positive normalized traces). We prove that every Weyl operator is strongly measurable, so spectral measurability implies strong measurability. The converse does not hold in general; a spectral characterization via Pietsch's correspondence is obtained in the companion paper~\cite{Po-Ti:Part2}. Finally, as an application, we establish spectral measurability for operators arising from nonclassical Weyl laws: operators associated with the Riemann Hypothesis (assuming RH); Schr\"odinger operators with anisotropic potentials and Dirichlet Laplacians on infinite-volume domains; Dirac operators on open spin manifolds with conformally cusp metrics; and the operator formed by the Dirac operator of the Podle\'s quantum sphere and the Laplacian on the 2-torus.
\end{abstract}
\maketitle

\tableofcontents

\section{Introduction}
In the framework of Connes' noncommutative geometry~\cite{Co:NCG}, the role of the integral is played by positive normalized traces, and more especially
Dixmier traces~\cite{Di:CRAS66}, on the weak trace class $\sL_{1,\infty}$. In the approach of~\cite{LSZ:Book, Po:JNCG23, Su:CIMPA17}, Dixmier traces are obtained by applying an \emph{extended limit} $\omega$ to the Ces\`aro mean, 
\begin{equation}\label{eq:Intro.Cesaro}
 \frac{1}{\log (N+1)}\sum_{j<N}\lambda_j(T), \qquad N\geq 0, 
\end{equation}
where $(\lambda_j(T))_{j\geq 0}$ is any eigenvalue sequence for $T$ (see Section~\ref{sec:Dixmier} for the precise meaning of eigenvalue sequence). 

Following Connes~\cite{Co:NCG}, $T$ is called \emph{measurable} (or \emph{Dixmier-measurable}) if $\Trw(T)$ is independent of the choice of $\omega$; this common value is the \emph{noncommutative integral} $\bint T$. It can further be shown (see~\cite{Co:NCG, LSZ:Book, Po:JNCG23}) that
\begin{equation}\label{eq:Intro.spectral-measurability-lunf}
 \biggl(\text{$T$ is measurable and}\ \bint T=L\biggr) \Longleftrightarrow \lim_{N\rightarrow \infty} \frac{1}{\log N}\sum_{j < N} \lambda_j(T)=L,
\end{equation}
giving both a spectral characterization of measurability and a practical way to compute NC integrals without reference to extended limits.

Stronger notions of measurability arise by enlarging the class of traces. From the point of view of noncommutative geometry, it is natural to use all positive normalized traces. The corresponding notion is called \emph{strong measurability} or \emph{positive measurability}. A spectral characterization of strongly measurable operators on $\sL_{1,\infty}$ was provided by Semenov-Sukochev-Usachev-Zanin~\cite{SSUZ:AIM15} as a consequence of a version of Pietsch's correspondence for traces on $\sL_{1,\infty}$. 

Birman-Solomyak~\cite{BS:JFAA70, BS:Book} developed a perturbation theory for eigenvalue and singular-value asymptotics of compact operators in weak Schatten classes $\sL_{p,\infty}$, $p>0$, which became a cornerstone of operator ideal techniques in semiclassical analysis. The existence of Weyl laws for operators in $\sL_{1,\infty}$ implies strong measurability (see, e.g.,~\cite{Po:JNCG23}). In particular, the Weyl laws of Birman-Solomyak~\cite{BS:VLU77, BS:VLU79, BS:SMJ79} for negative-order \psidos\ imply a stronger form of Connes' trace theorem~\cite{Co:CMP88}.

The aim of this paper, and its companion paper~\cite{Po-Ti:Part2}, is to extend this circle of ideas to the larger setting of \emph{weak Lorentz ideals}, with a special emphasis on the spectral aspects. This is motivated by instances of ``nonclassical'' Weyl laws in various settings, including Connes' approach to the Riemann Hypothesis (see below).

 \subsection{Nonclassical Weyl laws and weak Lorentz ideals}
 \emph{Classical Weyl laws} are eigenvalue asymptotics of the form,
 \begin{equation*}
 \lambda_j(T)\sim c j^{-\frac{1}{p}} \qquad \text{as}\ j\rightarrow \infty,
\end{equation*}
where $\lambda_0(T)\geq \lambda_1(T)\geq \cdots $ are the eigenvalues of $T$ (which for simplicity is assumed to be a positive compact operator). As coined by Simon~\cite{Si:JFA83},  \emph{nonclassical Weyl laws} are eigenvalue asymptotics of the form,
 \begin{equation*}
 \lambda_j(T)\sim c j^{-\frac{1}{p}}(\log j)^q,  \qquad q\neq 0.
\end{equation*}
 They appear in numerous settings, including Dirichlet problems on infinite-volume domains, Weyl pseudodifferential operators on Euclidean space, Weyl laws on open manifolds, Weyl laws in sub-Riemannian geometry, Weyl laws for bisingular operators, \emph{etc}.  (see, e.g., \cite{Ba:MZ12, CHT:arXiv22, DR:LNM87, GU:JMAA20, Mo:MA08, Si:JFA83} and the references therein).

Assuming the Riemann Hypothesis (RH) holds, one of the most striking examples of nonclassical Weyl laws occurs for the Dirac operator $\sD$ whose eigenvalues are
the imaginary parts of the zeros of Riemann's zeta function (see, e.g.,~\cite{Co:AFA24, CM:PNAS22}). The Riemann--von Mangoldt formula then implies that
\begin{equation*}
 \lambda_j\left(|\sD|^{-1}\right) \sim c j^{-1}\log j  \qquad \text{as}\ j\rightarrow \infty.
\end{equation*}
Therefore, at least from the point of view of noncommutative geometry and Connes' approach to RH, it is important to extend the framework of Connes' NC integral to operators satisfying nonclassical Weyl laws.

The right operator-ideal setting is provided by \emph{weak Lorentz ideals} (in the sense of~\cite{LSZ:Book}).  In this paper's setup they are associated with functions of regular variation of negative index, i.e., functions that are decreasing near $\infty$ and satisfy $\lim_{t\to\infty}g(\lambda t)/g(t)=\lambda^\rho$ for some $\rho<0$ (which is called the index of $g$). The corresponding weak Lorentz ideal is 
\begin{equation*}
 \sL_g=\left\{ T; \ \mu_j(T)=\op{O}(g(j))\right\},
\end{equation*}
where $\mu_0(T)\geq \mu_1(T)\geq \cdots $ are the singular values of $T$. In particular, for $g(t)=(t+1)^{-1/p}$, $p>0$, we recover the weak Schatten class $\sL_{p,\infty}$.

In this paper we develop a systematic theory of Dixmier traces and Connes' integration in the general setting of weak Lorentz ideals, extend the Birman--Solomyak perturbation theory, and set up the theory of strong measurability in this setting. We also provide several examples that are either new or refinements of previously known examples. Further results and examples are provided in the companion paper~\cite{Po-Ti:Part2}.  
We also refer to~\cite{GS:JFA14, GU:JMAA20, LU:JMAA25, SS:JFA13, TU:JMAA27}, and the references therein, for various related results.

\subsection{Measurability and Connes' integral}
Originally, Dixmier traces were defined on the Dixmier-Macaev ideal $\sM_{1,\infty}$, which strictly contains the weak trace class $\sL_{1,\infty}$ (see~\cite{Co:NCG, Di:CRAS66}). More recently (see \cite{LSZ:Book, Po:JNCG23, Su:CIMPA17}) it was observed that a simpler construction of Dixmier traces was available on the smaller weak trace class~$\sL_{1,\infty}$ as a direct consequence of the following asymptotic additivity property,
\begin{equation}\label{eq:Intro.asymptotic-additivity-lunf}
 \sum_{j<N} \lambda_j(S+T)= \sum_{j<N} \lambda_j(S) + \sum_{j<N} \lambda_j(T)+\op{O}(1), \qquad S,T\in \sL_{1,\infty}.
\end{equation}
The above result was established in the first edition of~\cite{LSZ:Book}. It readily implies that applying any extended limit to the Ces\`aro mean~(\ref{eq:Intro.Cesaro}) defines a positive linear trace on $\sL_{1,\infty}$. In particular, the Dixmier trace is defined directly on the whole ideal $\sL_{1,\infty}$ and the additivity is immediate.

It has been known for some time that the original approach to Dixmier traces of~\cite{Co:NCG, Di:CRAS66} extends \emph{mutatis mutandis} to Lorentz ideals (a.k.a.\ Marcinkiewicz ideals; see, e.g.,~\cite{DdPSS:Pos98, GS:JFA14, SS:JFA13} and the references therein). A key observation of this paper is that a version of~(\ref{eq:Intro.asymptotic-additivity-lunf}) holds for their weak versions as well.

From now on, we let $g:[0,\infty)\rightarrow (0,\infty)$ be a continuous $\RV_{-1}$-function such that $\int_0^\infty g(t)\,dt = \infty$, and set $G(t) = \int_0^t g(s)\,ds$.

\begin{propositionalph}[see Proposition~\ref{prop:asymptotic-additivity}] \label{lem:Intro.asymptotic-additivity}
 For all $S,T\in \sL_g$, we have
 \begin{equation}\label{eq:Intro.asymptotic-additivity}
  \sum_{j<N} \lambda_j(S+T)=  \sum_{j<N} \lambda_j(S) +  \sum_{j<N} \lambda_j(T)+\op{o}\left(G(N)\right).
\end{equation}
\end{propositionalph}

This result is a direct consequence of combining results of Kalton~\cite{Ka:Crelle98} and Lord-Sukochev-Zanin~\cite{LSZ:Book} with Karamata's theorem for $\RV_{-1}$-functions. It enables us to extend to weak Lorentz ideals the approach of~\cite{LSZ:Book, Po:JNCG23} to Dixmier traces and Connes' integral on $\sL_{1,\infty}$.

\begin{propositionalph}[see Proposition~\ref{prop:dixmier}]
Given any extended limit $\omega$, the following formula
\begin{equation}\label{intro:def-dixmier}
\Trw(T) := \omega\Bigl(\Bigl\{ \frac{1}{G(N)}\sum_{j < N}\lambda_j(T)\Bigr\}_{N \geq 0}\Bigr), \qquad T\in \sL_g,
\end{equation}
 defines a positive linear trace on $\sL_g$ which annihilates $(\sL_g)_0$.
\end{propositionalph}

As in the $\sL_{1,\infty}$-setting, we say that an operator $T\in \sL_g$ is \emph{measurable} if $\Trw(T)$ takes the same value for all extended limits $\omega$. This common value is then defined to be the \emph{noncommutative integral} of $T$, and is denoted $\gbint T$. The set of measurable operators is a closed subspace of $\sL_g$, which contains both its commutator space $\Com(\sL_g)$ and the separable ideal $(\sL_g)_0$ (see Proposition~\ref{prop:measurable-space-properties}).

We have the following extension of~(\ref{eq:Intro.spectral-measurability-lunf}) to weak Lorentz ideals.

\begin{theoremalph}[see Theorem~\ref{thm:measurable} for the full statement]\label{thm:intro1}
For any $T \in \sL_g$, we have
\begin{equation}\label{eq:Intro.spectral-measurable}
 \biggl(\text{$T$ is measurable and}\ \gbint T=L\biggr) \Longleftrightarrow \lim_{N\rightarrow \infty} \frac{1}{G(N)}\sum_{j < N} \lambda_j(T)=L .
\end{equation}
\end{theoremalph}

As with~(\ref{eq:Intro.spectral-measurability-lunf}) this provides us with a spectral characterization of measurable operators in $\sL_g$ and a spectral interpretation of $\gbint$. This also provides a simple test for measurability and a practical way to compute NC integrals in applications.

In particular, the asymptotic additivity property~(\ref{eq:Intro.asymptotic-additivity}) gives a direct proof that the r.h.s.\ of~(\ref{eq:Intro.spectral-measurable}) is a singular trace on its domain, answering a question of Connes~\cite[Question~A]{Po:JNCG23} in the generality of all weak Lorentz ideals (see \S\ref{subsec:Connes-question}).

\subsection{Birman-Solomyak perturbation theory for weak Lorentz ideals}
The perturbation theory of Birman-Solomyak~\cite{BS:JFAA70, BS:Book} for asymptotics of eigenvalues and singular values of operators in weak Schatten classes is a cornerstone of the use of operator ideal techniques in semiclassical analysis. We observe that the $\RV$-property is the very property that allows us to extend Birman-Solomyak's theory to weak Lorentz ideals.

Let $g:[0,\infty)\rightarrow (0,\infty)$ be a continuous $\RV_\rho$-function with $\rho<0$. We equip the quotient space $\dot{\sL}_g := \sL_g/(\sL_g)_0$ with its natural quotient quasi-norm, which satisfies $\|\dot{T}\|_{\dot{g}} = \limsup_{j\to\infty} g(j)^{-1}\mu_j(T)$ (see Lemma~\ref{lem:Lorentz.distance-formula}).

If $T=T^*\in \sL_g$ we let $\lambda_0^\pm (T)\geq \lambda_1^\pm (T)\geq \cdots $ be the positive/negative eigenvalues of $T$. We say that $T$ is a \emph{Weyl operator} if
\begin{equation*}
 \Lambda^+(T):= \lim_{j\rightarrow \infty} g(j)^{-1}\lambda_j^+(T) \quad \text{and}\quad  \Lambda^-(T):= \lim_{j\rightarrow \infty} g(j)^{-1}\lambda_j^{-}(T)\quad \text{both exist}.
\end{equation*}
More generally, if $T\in \sL_g$ is not selfadjoint, then $T$ is called a Weyl operator if its real and imaginary parts both are Weyl operators. We then define
\begin{equation*}
 \Lambda^{\pm}(T):=  \Lambda^{\pm}\left(\Re T\right)+ i \Lambda^{\pm}\left(\Im T\right).
\end{equation*}
In addition, we denote by $\sW(\sL_g)$ the set of Weyl operators in $\sL_g$.

The following perturbation result extends to all weak Lorentz ideals the perturbation result of Birman-Solomyak for eigenvalue asymptotics of operators in weak Schatten classes.

\begin{theoremalph}[see Theorem~\ref{thm:cont-Lambda}]\label{thm:Intro.Bir-Sol.Lambda}
 Let $T\in \sL_g$ and $T_\ell\in \sW(\sL_g)$, $\ell \geq 0$, be such that
 \begin{equation*}
 \dot{T}_\ell \longrightarrow \dot{T} \quad \text{in}\ \dot{\sL}_g.
\end{equation*}
Then $\lim_{\ell \rightarrow \infty} \Lambda^\pm(T_\ell)$ exist, $T\in \sW(\sL_g)$, and we have
\begin{equation*}
 \Lambda^\pm(T)= \lim_{\ell \rightarrow \infty} \Lambda^\pm(T_\ell).
\end{equation*}
 In particular, if $T$ and the $T_\ell$ are selfadjoint, then
 \begin{equation*}
 \lim_{j \rightarrow \infty} g(j)^{-1} \lambda_j^\pm(T)=  \lim_{\ell \rightarrow \infty} \big(\lim_{j \rightarrow \infty} g(j)^{-1} \lambda_j^\pm(T_\ell)\big).
\end{equation*}
\end{theoremalph}

This result implies that $\sW(\sL_g)$ is a closed subset of $\sL_g$ on which the functionals $\Lambda^\pm$ are continuous (Corollary~\ref{cor:Bir-Sol.continuity-Lambda}). It also implies that $\sW(\sL_g)$ and the functionals $\Lambda^\pm$ are invariant under perturbations by operators in $(\sL_g)_0$ (Corollary~\ref{cor:ptb-Lambda}). This extends to weak Lorentz ideals a well-known result of Weyl~\cite{We:NGWG11}.

We also have a version for singular values of Theorem~\ref{thm:Intro.Bir-Sol.Lambda} and its above-mentioned consequences (see Theorem~\ref{thm:Bir-Sol.Delta} and its corollaries). In particular, we obtain an extension to weak Lorentz ideals of Ky Fan's theorem (see Corollary~\ref{cor:Bir-Sol.Ky-Fan-Lorentz}).

Finally, as with weak Schatten classes  (see, e.g., \cite[\S11.6]{BS:Book}) all these perturbation results for eigenvalue asymptotics admit equivalent reformulations in terms of counting functions (see Proposition~\ref{prop:Bir-Sol.Npm} and Proposition~\ref{prop:Bir-Sol.nu}).

\subsection{Strong measurability}
Let $T_{g}$ be any positive operator in $\sL_g$ such that $\lambda_j(T_g)=g(j)$ for $j$ large enough. A trace $\varphi$ on $\sL_g$ is called \emph{$g$-normalized} if $\varphi(T_{g})=1$. For instance, every Dixmier trace is $g$-normalized.

As with the weak trace class $\sL_{1,\infty}$, we may consider stronger notions of measurability by considering larger classes of traces. From the perspective of noncommutative geometry, it is natural to consider the whole class of positive normalized traces. Accordingly, we shall say that an operator $T\in \sL_g$ is \emph{strongly measurable} (or positively measurable) if $T$ takes the same value on all positive normalized traces. In particular, such an operator is measurable. We denote by $\sM_s(\sL_g)$ the space of strongly measurable operators.

As normalized positive traces span the whole space of continuous traces, it can be shown (see Proposition~\ref{prop:decomp}) that
\begin{equation*}
 \sM_s(\sL_g)=\C T_{g} \oplus \overline{\Com(\sL_g)}.
\end{equation*}
Moreover, all the operators in $(\sL_g)_0$ are strongly measurable (see Proposition~\ref{prop:strong-meas.positive-trace-sLg0}). We also have the following result, which extends to weak Lorentz ideals the corresponding result for $\sL_{1,\infty}$ in~\cite{Po:JNCG23}.

\begin{theoremalph}[see Theorem~\ref{thm:weyl implies strong}] \label{thm:Intro.weyl implies strong}
 If $T\in \sL_g$ is a Weyl operator, then it is strongly measurable, and we have
 \begin{equation*}
 \gbint T = \Lambda^+(T)-\Lambda^{-}(T).
\end{equation*}
In particular, if $T$ is selfadjoint, then
\begin{equation*}
 \gbint T = \lim_{j\rightarrow \infty} g(j)^{-1} \lambda_j^+(T)-  \lim_{j\rightarrow \infty} g(j)^{-1} \lambda_j^-(T).
\end{equation*}
\end{theoremalph}

The above result relates Weyl law properties to measurability. It provides a simple way to check strong measurability. The converse does not hold.

This leads us to introduce a further notion of measurability. We shall call an operator \emph{spectrally measurable} if it is a Weyl operator in $\sL_g$. In particular, an operator $T=T^*\in \sL_g$ is spectrally measurable if and only if $\lim_{j\rightarrow \infty}g(j)^{-1} \lambda_j^\pm(T)$ both exist. Thus, Theorem~\ref{thm:Intro.weyl implies strong} asserts that spectral measurability implies strong measurability.

We also show that the space $\sM_s(\sL_g)$ does not depend on the choice of representative function in the equivalence class defining $\sL_g$ (Proposition~\ref{prop:strong-independent}).

In the companion paper~\cite{Po-Ti:Part2} we will prove a spectral characterization of strongly measurable operators on $\sL_g$. It will extend to weak Lorentz ideals the spectral characterization of strongly measurable operators on $\sL_{1,\infty}$ of Semenov-Sukochev-Usachev-Zanin~\cite{SSUZ:AIM15}. 

\subsection{Examples arising from nonclassical Weyl laws}
In Section~\ref{sec:examples}, we present several examples of spectrally measurable operators arising from various settings. Recall that by Theorem~\ref{thm:Intro.weyl implies strong} spectral measurability implies strong measurability. Therefore, in all the examples below the operators are measurable with respect to the whole set of positive normalized traces. 

We use a general paradigm to construct our examples. Namely, suppose that $A$ is a selfadjoint operator with non-negative spectrum such that $0$ is isolated in the spectrum, and the positive part of the spectrum consists of positive eigenvalues with finite multiplicity. Assume further that the counting function $N(A;\lambda)$ of $A$ satisfies a (nonclassical) Weyl law of the form,
\begin{equation}\label{eq:Intro.A-Weyl-law}
 N(A;\lambda)\sim c \lambda^p (\log \lambda)^q, \qquad c>0, \quad p>0, \quad q\geq -1.
\end{equation}
\begin{lemmalph}[see Lemma~\ref{lem:example.Ap-spectral-meas}] \label{lem:Intro.Ap-spectral-meas}
If $A$ satisfies the Weyl law~(\ref{eq:Intro.A-Weyl-law}), then the operator $A^{-p}$ is spectrally measurable in $\sL_g$, and we have
 \begin{equation*}
 \gbint A^{-p} = cp^{-q}.
\end{equation*}
\end{lemmalph}

The specific examples treated in Section~\ref{sec:examples} include the following:
\begin{itemize}
\item \emph{Riemann Hypothesis.} Assuming (RH), let $\sD$ be the Dirac operator whose eigenvalues are the imaginary parts of the zeros of $\zeta(s)$ on the critical line $\Re s = 1/2$ (see, e.g.,~\cite{Co:AFA24, CM:PNAS22}). The Riemann--von Mangoldt formula implies the nonclassical Weyl law,
\[
N(|\sD|;\lambda) = \frac{1}{\pi}\,\lambda\log\lambda\Bigl(1 + \op{O}\bigl((\log\lambda)^{-1}\bigr)\Bigr).
\]
Setting $g(t) = (t+1)^{-1}\log(t+2)$, the operator $|\sD|^{-1}$ then is spectrally measurable in $\sL_g$ (see Proposition~\ref{prop:RH}).\smallskip

\item \emph{Schr\"odinger operators and Dirichlet Laplacians on unbounded domains.} On $\R^n$, $n\geq 2$, consider the Schr\"odinger operators, 
\begin{equation*}
 H_\alpha := -\Delta + |x_1\cdots x_n|^\alpha, \qquad \alpha>0. 
\end{equation*}
These operators were introduced by Simon~\cite{Si:AP83, Si:JFA83} in 2D and studied in higher dimensions in~\cite{CR:JMP15}. It is shown in~\cite{CR:JMP15,Si:JFA83} that we have nonclassical Weyl laws, 
\[
N(H_\alpha;\lambda) \sim c(n,\alpha)\,\lambda^{\frac{n}{2}+\frac{1}{\alpha}}(\log\lambda)^{n-1},
\]
where $c(n,\alpha)$ is some explicit constant (see Eq.~(\ref{eq:Examples.Simon-cnalpha})). It was actually this Weyl law that prompted Simon~\cite{Si:JFA83} to coin the terminology ``nonclassical Weyl law''. Thus, if we set $g(t) = (t+1)^{-1}(\log(t+2))^{n-1}$, then the operator $H_\alpha^{-(n/2+1/\alpha)}$ is spectrally measurable in $\sL_g$ (see Proposition~\ref{prop:Simon}). 

As noticed in~\cite{CR:JMP15, Si:JFA83}, as $\alpha \rightarrow \infty$ the spectrum of $H_\alpha$ is closely related to that of the  Dirichlet Laplacian $\Delta_\Omega$ on the infinite-volume domain, 
 \[
 \Omega:=\{(x_1,\ldots, x_n)\in \R^n;\ |x_1\cdots x_n|<1\}. 
\]
In particular, as shown in~\cite{CR:JMP15, Si:JFA83}, we have the Weyl law, 
\[
N\big(\Delta_\Omega;\lambda\big) \sim c(n,\infty)\,\lambda^{\frac{n}{2}} (\log \lambda)^{n-1}, \qquad c(n,\infty):=\frac{2n^n}{n!}(2\pi)^{-n}|\bB^n|.
\]
This implies that  $\Delta_\Omega^{-n/2}$ is likewise spectrally measurable in $\sL_g$ (see Proposition~\ref{prop:Examples.Dirichlet}). This refines~\cite[Example~5.5]{TU:JMAA27}, where only Dixmier-measurability was established (see also~\cite{GU:JMAA20}).\smallskip

\item \emph{Open manifolds with conformally cusp metrics.} Let $X^n$ be an open spin manifold with a conformally cusp metric $\bar{g}=x^{2r}g_0$ (examples include finite-volume hyperbolic manifolds and metric horns). Assuming no harmonic spinors on the boundary, Moroianu~\cite{Mo:MA08} proved that $\shD_{\bar{g}}$ has discrete spectrum, with Weyl laws,
\[
N(|\shD_{\bar{g}}|;\lambda) \sim
\begin{cases}
\slashed{c}_1(n)\,\Vol_{\bar{g}}(X)\,\lambda^n & \text{if}\ r > 1/n,\\[4pt]
\slashed{c}_2(n)\,\Vol_{h_0}(M)\,\lambda^n\log\lambda & \text{if}\ r = 1/n.
\end{cases},
\]
where $ \slashed{c}_1(n)$ and $\slashed{c}_2(n)$ are explicit constants (see Eq.~(\ref{eq:Examples.Open-constants})). This gives spectral measurability of $|\shD_{\bar{g}}|^{-n}$ in $\sL_{1,\infty}$ when $r>1/n$, and in $\sL_g$ with $g(t)=(t+1)^{-1}\log(t+2)$ when $r=1/n$. In the latter case, we further have
\begin{equation*}
 \gbint|\shD_{\bar{g}}|^{-n}=\slashed{c}_2(n)\Vol_{h_0}(M).
\end{equation*}
This shows that the NC integral recaptures the volume of the boundary $M=\partial X$, even though $(X^n,g)$ has infinite volume for $r=1/n$ (see Proposition~\ref{prop:cusp}).\smallskip

\item \emph{Podle\'s quantum spheres.} Given $q\in(0,1)$, let $(\cO(\bS^2_q),\sH,\sD_q)$ be the spectral triple of Dabrowski--Sitarz~\cite{DS:PAN03} on the Podle\'s quantum sphere, which has ``dimension zero'' in the sense of noncommutative geometry (see~\cite{EIS:CMP14}). Following Gayral--Sukochev~\cite{GS:JFA14}, we consider the operator, 
\begin{equation*}
 A:=|\sD_q|\otimes 1 + 1\otimes \Delta, 
\end{equation*}
where $\Delta$ is the Laplacian on $\T^2=(\R/(2\pi\Z))^2$. Using the results of~\cite{EIS:CMP14, GS:JFA14} it can be shown that we have the nonclassical Weyl law,
\[
N\bigl(A;\lambda\bigr)\sim\frac{\pi}{(\log q)^2}\,\lambda(\log\lambda)^2.
\]
It follows that $(|\sD_q|\otimes 1 + 1\otimes\Delta)^{-1}$ is spectrally measurable in $\sL_g$ with $g(t)=(t+1)^{-1}(\log(t+2))^2$ (see Proposition~\ref{prop:Podles}). This  improves the measurability result of~\cite{GS:JFA14} where only Dixmier-measurability in a larger ideal was established (see also~\cite{GU:JMAA20, TU:JMAA27}). 
\end{itemize}

We refer to~\cite{DR:LNM87, GU:JMAA20, LU:JMAA25, TU:JMAA27}, and the companion paper~\cite{Po-Ti:Part2}, for further instances of nonclassical Weyl laws. Combining these nonclassical Weyl laws with Lemma~\ref{lem:Intro.Ap-spectral-meas} produces further examples of spectrally measurable operators in weak Lorentz ideals. 

\subsection{Organization of the paper}
This paper is organized as follows.
In Section~\ref{sec:Lorentz}, we recall background on singular values, quasi-Banach ideals, functions of regular variation, and weak Lorentz ideals.
In Section~\ref{sec:Dixmier} we construct Dixmier traces on weak Lorentz ideals, establish the spectral characterization of measurable operators, and provide the spectral-theoretic construction of Connes' integral.
In Section~\ref{sec:Bir-Sol} we develop the Birman--Solomyak perturbation theory for eigenvalues and singular values of operators in $\sL_g$.
In Section~\ref{sec:strong-hyper} we study strong measurability and spectral measurability, and show that the latter implies the former. We also examine the independence of strong measurability from the choice of the function $g(t)$. In Section~\ref{sec:examples} we describe the concrete examples arising from nonclassical Weyl laws mentioned above, including the Riemann Hypothesis example.
Appendix~\ref{app:proofs} collects proofs of technical results for weak Lorentz ideals. Appendix~\ref{app:counting} develops the counting function asymptotics used in Sections~\ref{sec:Bir-Sol} and~\ref{sec:strong-hyper}.

\subsection*{Acknowledgments} This paper originated from discussions with Alexander Usachev. We would like to thank him warmly for sharing his insights with us. We are also grateful to Grigori Rozenblum for bringing~\cite{Ro:MUS74} to our attention. The first-named author also wishes to thank Magnus Goffeng and Edward McDonald for discussions related to the subject matter of this paper. His research was partially supported by NSFC grant No.~11971328 (China).

\section{Singular Values and Weak Lorentz Ideals}\label{sec:Lorentz}
In this section, we record the main definitions and properties regarding weak Lorentz ideals.
Throughout this paper we let $\sH$ be a (separable) Hilbert space with inner product $\scal{\cdot}{\cdot}$. The algebra of bounded linear operators on $\sH$ is denoted $\sL(\sH)$. The operator norm is denoted $\|\cdot\|$. We also denote by $\sK$ the (closed) ideal of compact operators on $\sH$.

\subsection{Singular values}
Given any operator $T\in \sK$, we let $(\mu_j(T))_{j\geq 0}$ be its sequence of \emph{singular values}, i.e., $\mu_j(T)$ is the $(j+1)$-th eigenvalue counted with multiplicity of the absolute value $|T|=\sqrt{T^*T}$. By the \emph{min-max principle} we have
\begin{align}
 \mu_j(T)=\min \left\{\|T_{|E^\perp}\|;\ \dim E=j\right\}.
  \label{eq:min-max}
\end{align}
In addition, we have
\begin{equation*}
  \mu_j(T)=\dist(T,\sR_j), \qquad \text{where}\ \sR_j=\{R\in \sL(\sH);\ \rk R\leq j\}.
\end{equation*}

We record the following properties of singular values (see, e.g., \cite{BS:Book, GK:AMS69}),
\begin{gather}
 \mu_j(T)=\mu_j(T^*)=\mu_j(|T|),
 \label{eq:Quantized.properties-mun1}\\
 \left|\mu_j(T)-\mu_j(S)\right|\leq \|S-T\|,
 \label{eq:Quantized.properties-mun2}\\
\mu_j(ATB)\leq \|A\| \mu_j(T) \|B\| \qquad \forall A, B\in \sL(\sH),
 \label{eq:Quantized.properties-mun3}\\
 \mu_j(U^*TU)= \mu_j(T)
 \qquad \forall U\in\sL(\sH), \ \text{$U$ unitary}.
  \label{eq:Quantized.properties-mun4}
\end{gather}
In addition, we have Ky Fan's inequalities,
\begin{gather}
  \mu_{j+k}(S+T)\leq \mu_j(S) + \mu_k(T),
 \label{eq:Lorentz.Fan1}\\
 \sum_{j<N} \mu_j(S+T) \leq \sum_{j<N} \mu_j(S) +  \sum_{j<N} \mu_j(T).
 \label{eq:Lorentz.Fan2}
 \end{gather}

\subsection{Quasi-Banach ideals}
We briefly recall basic definitions regarding quasi-Banach ideals.

\begin{definition}
A \emph{quasi-norm} on a vector space $E$ is a function $\|\cdot\|: E\rightarrow [0,\infty)$ with the following properties:
\begin{enumerate}
 \item[(i)] $\| \lambda x\|=|\lambda| \|x\|$ for all $x\in E$ and $\lambda \in \C$.

 \item[(ii)] $\|0\|=0$ and $\|x\|>0$ for $x\neq 0$.

  \item[(iii)] There exists $C>0$ such that
 \begin{equation}\label{eq:quasi}
 \|x+y\|\leq C\left(\|x\|+\|y\|\right) \qquad \forall x,y\in E.
\end{equation}
\end{enumerate}
\end{definition}

In other words, a quasi-norm is like a norm where the usual triangle inequality is relaxed into the quasi-triangular inequality. In particular, any norm is a quasi-norm. Similarly to norms, any quasi-norm $\|\cdot\|$ on a vector space $E$ defines a topology in which a basis of the neighborhood system for the origin consists of the balls
\begin{equation*}
 B(0,\delta)=\left\{x\in E; \ \|x\|\leq \delta\right\}, \qquad \delta>0.
\end{equation*}
Note this topology need not be locally convex if $\|\cdot\|$ is not a norm. Moreover, every point admits a countable neighborhood basis.

We say that a quasi-norm on $E$ is \emph{$r$-convex} with $r\in (0,1]$ if
\begin{equation*}
 \|x+y\|^r \leq \|x\|^r +\|y\|^r \qquad \forall x,y \in E.
\end{equation*}
 This property of $r$-convexity implies the quasi-triangular inequality with a specific constant:
\begin{equation*}
 \|x+y\| \leq 2^{\frac{1-r}{r}}\big(\|x\|+\|y\|\big), \qquad x,y\in E.
\end{equation*}
Moreover, the $r$-convexity implies that the function $d(x,y):=\|x-y\|^r$, $x,y\in E$, is a metric that defines the same topology as the original quasi-norm $\|\cdot\|$. Conversely, any quasi-norm is equivalent to an $r$-convex quasi-norm for some  $r\in (0,1]$ (see, e.g.,~\cite[Theorem~1.3]{KPR:Cam84}). As a result any quasi-norm topology is metrizable.

\begin{definition}
 A \emph{quasi-Banach space} is a quasi-normed space whose topology is complete, i.e., every Cauchy sequence is convergent.
\end{definition}

\begin{remark}
 The quasi-triangular inequality~(\ref{eq:quasi}) implies the inequalities,
 \begin{equation*}
 \| x_1+\cdots + x_n\| \leq \sum_{1\leq j \leq n} C^j \|x_j\|, \qquad x_j \in E.
\end{equation*}
Thus, if $E$ is a quasi-Banach space, and we have
\begin{equation*}
\sum_{j\geq 0} C^j \|x_j\|<\infty, \qquad x_j\in E,
\end{equation*}
then the series $\sum x_j$ converges in $E$ (see~\cite[Lemma 3.2.3]{LSZ:Book}). Conversely, this property characterizes the completeness of a quasi-norm
(see the proof of~\cite[Theorem~3.1.2]{LSZ:Book}; see also~\cite[Theorem~5.1]{Fo:RealAnalysis} for a proof in the setting of normed vector spaces).
\end{remark}

\begin{definition}
 A \emph{quasi-Banach ideal} is a two-sided ideal $\sJ$ with a complete quasi-norm $\|\cdot\|_{\sJ}$ such that
 \begin{equation*}
 \|ATB\|_{\sJ} \leq \|A\| \|T\|_{\sJ} \|B\| \qquad \text{for all $T\in \sJ$ and $A,B\in \sL(\sH)$}.
\end{equation*}
If $\|\cdot\|_\sJ$ is a norm, then we say that $\sJ$ is a \emph{Banach ideal.}
\end{definition}

\begin{example}
If $p>0$, then the \emph{Schatten class},
\begin{equation*}
 \sL_p:= \big\{ T\in \sK; \ \sum \mu_j(T)^p<\infty\},
\end{equation*}
is a quasi-Banach ideal with quasi-norm,
\begin{equation*}
 \|T\|_p:=  \big(\sum \mu_j(T)^p\big)^{1/p}, \qquad T\in \sL_p.
\end{equation*}
This is actually a Banach ideal for $p\geq 1$.
\end{example}

\begin{example}
For $p>0$, the \emph{weak Schatten class},
\begin{equation*}
 \sL_{p,\infty}:= \big\{ T\in \sK; \ \mu_j(T)=\op{O}(j^{-\frac{1}{p}})\big\},
\end{equation*}
is a quasi-Banach ideal with quasi-norm,
\begin{equation*}
 \|T\|_{p,\infty}:= \sup_{j\geq 0}(j+1)^{\frac 1p}\mu_j(T), \qquad T\in \sL_{p,\infty}.
\end{equation*}
 \end{example}

\begin{remark}\label{rmk:Lorentz.polar-decomposition}
 Every two-sided ideal $\sJ$ of $\sL(\sH)$ is contained in $\sK$ and contains the ideal of finite-rank operators. Moreover, properties of the polar decomposition imply that
\begin{equation*}
 \big( T\in \sJ \Longleftrightarrow |T|\in \sJ\big) \qquad \text{and} \qquad  \big( T\in \sJ \Longleftrightarrow T^*\in \sJ\big).
\end{equation*}
 If in addition $\sJ$ is a quasi-Banach ideal, then~(\ref{eq:Quantized.properties-mun1}) implies that
 \begin{gather*}
 \|T\|_{\sJ}=\|T^*\|_{\sJ}=\||T|\|_{\sJ},\\
 \|U^*TU\|_{\sJ}= \|T\|_{\sJ} \qquad \forall U\in \sL(\sH), \ U \ \text{unitary}.
\end{gather*}
In fact, if $S,T\in \sJ$, then
\begin{equation*}
\big( \mu_j(S)=\mu_j(T) \quad \forall j\geq 0 \big)\ \Longrightarrow \ \|S\|_{\sJ}=\|T\|_{\sJ}.
\end{equation*}
More generally, it can be shown that
\begin{equation*}
\big(T\in \sJ \quad \text{and} \quad \mu_j(S)\leq \mu_j(T) \quad \forall j\geq 0 \big)\ \Longrightarrow \ \big( S\in \sJ \quad \text{and} \quad \|S\|_{\sJ}\leq \|T\|_{\sJ}\big).
\end{equation*}
\end{remark}

\begin{remark}
We refer to~\cite{GK:AMS69, LSZ:Book, Si:AMS05} for further background on Banach ideals and quasi-Banach ideals.
\end{remark}

\subsection{Functions of regular variation}
In what follows, we shall say that a real-valued function $g(t)$ defined on an interval $[a,\infty)$ is \emph{ultimately increasing} (resp., \emph{ultimately decreasing}) if it is increasing (resp., decreasing) on some interval $[b,\infty)$ with $b\geq a$.

\begin{definition}
 An $L^\infty_\loc$-function $g:[a,\infty)\rightarrow (0,\infty)$ is \emph{slowly varying} if
 \begin{equation}\label{eq:slowly varying}
 \lim_{t\rightarrow \infty} \frac{g(\lambda t)}{g(t)}=1 \qquad \forall \lambda>0.
\end{equation}
\end{definition}

\begin{remark}
 If $g(t)$ is ultimately increasing or decreasing, then~(\ref{eq:slowly varying}) is equivalent to
 \begin{equation*}
  \lim_{t\rightarrow \infty} \frac{g(2 t)}{g(t)}=1.
\end{equation*}
\end{remark}

\begin{remark}[see~{\cite[Proposition~1.3.6]{BGT:Cambridge87}}]\label{rmk:sv vs power function}
 If $g(t)$ is slowly varying, then, for all $\rho>0$, we have
 \begin{equation*}
 \lim_{t\rightarrow \infty} t^\rho g(t)=\infty \qquad \text{and}\qquad \lim_{t\rightarrow \infty} t^{-\rho} g(t)=0.
\end{equation*}
\end{remark}

\begin{definition}
 Let $g:[a,\infty)\rightarrow (0,\infty)$ be an $L^\infty_\loc$-function. For $\rho \in \R$, we say that $g(t)$ has \emph{regular variation of index} $\rho$, or is $\RV_\rho$, if the following conditions are satisfied:
\begin{enumerate}
 \item[(i)] $g(t)$ is ultimately monotonic.

 \item[(ii)] We have
\begin{equation}\label{cond:rv}
 \lim_{t\rightarrow \infty} \frac{g(\lambda t)} {g(t)}= \lambda^\rho \qquad \forall \lambda>0.
\end{equation}
\end{enumerate}
 \end{definition}

\begin{remark}
 If $\rho>0$ (resp., $\rho<0$), then (i)--(ii) imply that $g(t)$ is ultimately increasing (resp., decreasing).
\end{remark}

\begin{remark}
 In our terminology an $\RV_0$-function is a slowly varying function that is either ultimately increasing or ultimately decreasing.
\end{remark}

\begin{remark}
 An $L^\infty_\loc$-function $g:[a,\infty)\rightarrow (0,\infty)$ is $\RV_\rho$ if it is ultimately monotonic and $t^{-\rho}g(t)$ is slowly varying. In particular, it follows from Remark~\ref{rmk:sv vs power function} that we have
\begin{equation}\label{eq:asymptotic-RV-power}
 \lim_{t\rightarrow \infty} t^{-\alpha} g(t)=\infty \quad \text{if $ \alpha <\rho$}, \qquad \textup{and}\qquad
 \lim_{t\rightarrow \infty} t^{-\alpha} g(t)=0 \quad \text{if $ \alpha >\rho$}.
\end{equation}
This implies that
\begin{equation*}
 \lim_{t\rightarrow \infty} g(t)=0 \quad \text{if $\rho<0$}, \qquad \textup{and}\qquad
 \lim_{t\rightarrow \infty} g(t)=\infty \quad \text{if $ \rho >0$}.
\end{equation*}
In addition, we have
\begin{equation*}
 \int_0^\infty g(t)dt <\infty \quad \text{if $\rho<-1$}, \qquad \text{and} \qquad \int_0^\infty g(t)dt =\infty \quad \text{if $\rho>-1$}.
\end{equation*}
 \end{remark}

\begin{remark}\label{rmk:UCT}
By the Uniform Convergence Theorem (UCT) (see, e.g., \cite[Theorem~1.5.2]{BGT:Cambridge87}) the convergence~(\ref{cond:rv}) holds uniformly with respect to $\lambda$ as it varies within compact subsets of $(0,\infty)$. If $\rho>0$ (resp., $\rho<0$), we even have uniform convergence on each interval $(0,a]$ (resp., $[a,\infty)$).
\end{remark}

Throughout this paper we will make frequent use of the following consequence of the UCT.

\begin{lemma}\label{lem:UCT-consequence}
 Let $g:[a,\infty)\rightarrow (0,\infty)$ be $\RV_\rho$, $\rho\in \R$.
 \begin{enumerate}
 \item If $t_\alpha \rightarrow \infty$ and $\lambda_\alpha \rightarrow \lambda>0$, then $g(\lambda_\alpha t_\alpha)\sim \lambda^\rho g(t_\alpha)$.

 \item If $\lim u(s)=\lim v(s)=\infty$ and $u(s)\sim v(s)$ for $s$ large, then $g(u(s))\sim g(v(s))$.
 \end{enumerate}
 \end{lemma}

\begin{remark}
 Part (2) follows from part (1) by taking $t_\alpha=u(s)$ and $\lambda_\alpha =v(s)/u(s)$.
\end{remark}

\begin{remark}\label{rmk:slow decreasing}
 As a first consequence of Lemma~\ref{lem:UCT-consequence}, we see that if $g:[a,\infty)\rightarrow (0,\infty)$ is $\RV_\rho$ , then
 \begin{equation}\label{eq:slow decreasing}
 \lim_{t\rightarrow \infty} \frac{g(t+b)}{g(t)}=1 \qquad \forall b\in \R.
\end{equation}
\end{remark}

\begin{example}
 The following are among the most well-known examples of $\RV_\rho$-functions on $[0,\infty)$.
 \begin{itemize}
 \item $g(t)=t^\rho$, $\rho>0$, and more generally $g(t)=(t+a)^\rho$, $\rho\neq 0$ with $a>0$.

 \item $g(t)=(\log(t+2))^q$, $q\neq 0$ (these functions are $\RV_0$).

 \item $g(t)=t^\rho (\log (t+2))^q$, $\rho \neq 0$, $q\neq 0$.
\end{itemize}
 \end{example}

\begin{theorem}[Karamata; see, e.g., {\cite{BGT:Cambridge87, BIKS:Springer18}}]\label{thm:karamata}
 Let $g:[a,\infty)\rightarrow (0,\infty)$ be $\RV_\rho$, $\rho\in \R$.
 \begin{enumerate}
 \item If $\rho>-1$, or if $\rho=-1$ and $\int_0^\infty g(t)dt=\infty$, then
\begin{equation}\label{eq:asymptotic--karamata}
 \lim_{t\rightarrow \infty} \frac{tg(t)}{\int_0^t g(s)ds} = \rho+1.
\end{equation}

\item If $\rho<-1$, or if $\rho=-1$ and $\int_0^\infty g(t)dt<\infty$, then
\begin{equation*}
 \lim_{t\rightarrow \infty} \frac{tg(t)}{\int_t^\infty g(s)ds} = |\rho+1|.
\end{equation*}
\end{enumerate}
\end{theorem}

\begin{remark}
 We refer to~\cite{BGT:Cambridge87, BIKS:Springer18} for further background on slowly varying functions and functions of regular variation.
\end{remark}

\begin{remark}\label{rmk:continuous rv}
 In the rest of the article we will solely consider \emph{continuous} $\RV$-functions. On the one hand, in practice the main examples of $\RV$-functions that show up are continuous and even smooth. On the other hand, in this paper, we are only interested in the values of these functions at integer points. It follows from a result of Adamovi\'c (see~\cite[Proposition~1.3.4]{BGT:Cambridge87}) that if $h(t)$ is an $\RV$-function we can always find an ultimately monotonic \emph{smooth} function $h_1(t)$ such that $h_1(t)\sim h(t)$ and $h_1=h$ on $\N_0$. Therefore, there is no loss of generality in restricting ourselves to continuous (or even smooth) $\RV$-functions.
\end{remark}

\subsection{Weak Lorentz ideals} In what follows, we let $g:[0,\infty)\rightarrow (0,\infty)$ be a continuous $\RV_\rho$-function with $\rho<0$. We also define
\begin{equation*}
 G(t):=\int_0^t g(s)ds, \qquad t\geq 0.
\end{equation*}
In addition, given any operator $T\in \sK$, we set
\begin{equation*}
 \|T\|_g := \sup_{j\geq 0} g(j)^{-1}\mu_j(T).
\end{equation*}

\begin{lemma}\label{lem:Lorentz.quasi-norm}
 The following hold.
 \begin{enumerate}
 \item[(i)] Let $T\in \sK$. We have
 \begin{gather*}
 \|\lambda T\|_g = |\lambda| \, \|T\|_g \qquad \forall \lambda \in \C,\\
 \|ATB\|_g \leq \|A\| \|T\|_g \|B\|   \qquad \forall A,B\in \sL(\sH).
\end{gather*}

\item[(ii)]  There is $C>0$ such that
\begin{equation}\label{eq:g-triangle inequality}
 \|S+T\|_g \leq C\big(\|S\|_g +\|T\|_g\big) \qquad \forall S,T\in \sK.
\end{equation}
\end{enumerate}
\end{lemma}

\begin{remark}\label{rmk:quasi-norm}
The first part is a direct consequence of~(\ref{eq:min-max}) and~(\ref{eq:Quantized.properties-mun3}). The second part is a standard consequence of Ky Fan's inequality~(\ref{eq:Lorentz.Fan1}).  For the reader's convenience a proof of this result is given in Appendix~\ref{app:proofs}.
\end{remark}

\begin{definition}[see~\cite{LSZ:Book}]
The  \emph{weak Lorentz ideal} associated with $g$ is
 \begin{equation*}
 \sL_g:=\left\{T \in \sK; \ \mu_j(T)= \op{O}\left(g(j)\right)\right\}.
\end{equation*}
\end{definition}

Lemma~\ref{lem:Lorentz.quasi-norm} ensures that $\sL_g$ is a two-sided ideal on which $\|\cdot\|_g$ is a quasi-norm satisfying~(\ref{eq:g-triangle inequality}). We actually have the following result.

\begin{proposition}\label{prop:Lorentz.quasi-Banach}
 $\sL_g$ is a quasi-Banach ideal.
\end{proposition}

\begin{remark}
 The only part that needs to be proved is the completeness of the quasi-norm $\|\cdot\|_g$. This can be deduced from Calkin's correspondence for quasi-Banach ideals as described in~\cite[\S 3]{LSZ:Book}. For the reader's convenience, a direct proof of this property is given in~Appendix~\ref{app:proofs}.
\end{remark}

\begin{example}
 If $g(t)=(1+t)^{-1/p}$, $p>0$, then $\sL_g$ is the weak Schatten class $\sL_{p,\infty}$.
\end{example}

\begin{remark}
 If $h(t)$ is a continuous $\RV$-function such that $h(t)=\op{O}(g(t))$, then we have a continuous inclusion $\sL_{g}\subseteq \sL_h$. In particular, in view of~(\ref{eq:asymptotic-RV-power}) and the previous example, we have continuous inclusions,
 \begin{equation*}
 \sL_{q,\infty} \subseteq \sL_g \subseteq \sL_{p,\infty} \qquad \text{for}\ q<|\rho|^{-1}<p.
\end{equation*}
\end{remark}

\begin{remark}
 The previous remark also implies that if $h(t)$ is a continuous $\RV_\rho$-function such that $h(t)=\op{O}(g(t))$ and $g(t)=\op{O}(h(t))$, e.g., if $h(t)\sim g(t)$, then $\sL_h=\sL_g$ with equivalent quasi-norms.
\end{remark}

In what follows we denote by $\sR$ the ideal of finite-rank operators. We set
\begin{equation*}
 \dist_{\sL_g}(T;\sR):= \inf\left\{\|T-R\|_g; \ R\in \sR\right\}, \qquad T\in \sL_g.
\end{equation*}
Recall also that if $T\in \sK$ has polar decomposition $T=U|T|$ and $(\xi_j)_{j\geq 0}$ is any orthonormal sequence in $\sH$ such that $|T|\xi_j= \mu_j(T)\xi_j$, we have the Schmidt series representations,
\begin{equation*}
 |T| = \sum_{j \geq 0} \mu_j(T)\ketbra{\xi_j}{\xi_j} \qquad \text{and} \qquad T=\sum_{j \geq 0}  \mu_j(T)\ketbra{U\xi_j}{\xi_j},
\end{equation*}
where the series converges in $\sK$. Here we use the ketbra notation $\ketbra{\eta}{\xi}$ to denote the projection onto $\C \eta$ along $(\C\xi)^\perp$.

We have the following formulas for the distance function $\dist_{\sL_g}(\cdot;\sR)$.

\begin{lemma}\label{lem:Lorentz.distance-formula}
 Let $T\in \sL_g$ have Schmidt series  $T=\sum_{j \geq 0} \mu_j(T)\ketbra{U\xi_j}{\xi_j}$. Set
 \begin{equation*}
T_N=\sum_{j<N} \mu_j(T) \ketbra{U\xi_j}{\xi_j},\quad N\geq 1.
 \end{equation*} We then have
 \begin{equation*}
  \dist_{\sL_g}(T;\sR)= \limsup_{N\rightarrow \infty} \|T-T_N\|_g = \limsup_{j\rightarrow \infty} g(j)^{-1}\mu_j(T).
\end{equation*}
\end{lemma}
\begin{proof}
 See Appendix~\ref{app:proofs}.
\end{proof}

\begin{remark}
 A version of the above result for weak Schatten classes is proved in~\cite{SXZ:JFA23}.
\end{remark}

In what follows we denote by $(\sL_{g})_0$ the closure of $\sR$ in $\sL_g$. This is the maximal closed sub-ideal of $\sL_g$. Moreover, we have
\begin{equation*}
 T\in (\sL_{g})_0 \Longleftrightarrow \dist_{\sL_g}(T;\sR)=0.
\end{equation*}
Therefore, Lemma~\ref{lem:Lorentz.distance-formula} provides the following characterization of $(\sL_{g})_0$.

\begin{proposition}\label{prop:separable part}
 Let $T\in \sK$. The following are equivalent:
 \begin{enumerate}
 \item[(i)] $T\in (\sL_{g})_0$.

 \item[(ii)] Every Schmidt series for $T$ converges in $\sL_g$.

 \item[(iii)] $\lim_{j\rightarrow \infty} g(j)^{-1}\mu_j(T)=0$.
\end{enumerate}
In particular, we have
\begin{equation*}
 (\sL_{g})_0= \left\{T \in \sK; \ \mu_j(T)= \op{o}\left(g(j)\right)\right\} \subsetneq \sL_g.
\end{equation*}
\end{proposition}

\begin{remark}
 The fact that $(\sL_{g})_0\subsetneq \sL_g$ implies that $\sL_g$ is not separable.
\end{remark}

For $T\in \sK$, define
\begin{equation}\label{eq:Lorentz norm}
 \|T\|_{G}:= \sup_{N\geq 1} \frac{1}{G(N)} \sum_{j<N} \mu_j(T).
\end{equation}
It follows from~(\ref{eq:min-max}), (\ref{eq:Quantized.properties-mun3}), and~(\ref{eq:Lorentz.Fan2}) that we have
 \begin{gather*}
 \|\lambda T\|_{G} = |\lambda| \, \|T\|_{G} \qquad \forall \lambda \in \C,\\
 \|ATB\|_{G} \leq \|A\| \|T\|_{G} \|B\|   \qquad \forall A,B\in \sL(\sH),\\
 \|S+T\|_{G} \leq \|S\|_{G} + \|T\|_{G} \qquad \forall S,T\in \sK.
 \end{gather*}

The \emph{Lorentz ideal}\footnote{Lorentz ideals are also called Marcinkiewicz ideals in the literature (see the discussion in~{\cite[page 75]{LSZ:Book}} on this topic).} associated with $G(t)$  is
\begin{equation*}
 \sM_{G}= \bigg\{T\in \sK;\ \sum_{j<N} \mu_j(T)=\op{O}(G(N))\bigg\}.
\end{equation*}
The functional $\|\cdot\|_{G}$ is a norm on $\sM_{G}$ with respect to which $\sM_{G}$ is a Banach ideal.

\begin{proposition}\label{prop:Lorentz.Marcinkiewicz}
 The following hold.
 \begin{enumerate}
 \item[(i)] We always have a continuous inclusion,
 \begin{equation*}
 \sL_g \subseteq \sM_{G}.
\end{equation*}

\item[(ii)] If $-1<\rho<0$, then $\sL_g=\sM_{G}$ with equivalent quasi-norms. In particular, $\sL_g$ is a quasi-Banach ideal with respect to the norm $\|\cdot\|_{G}$.
\end{enumerate}
\end{proposition}

\begin{remark}
 The above result, or at least part of it, is well known (see, e.g.,~\cite[\S 14]{GK:AMS69}). For the reader's convenience, a proof of this result is given in Appendix~\ref{app:proofs}.
 \end{remark}

\begin{example}
 If $g(t)=(1+t)^{-1}$, then $\sM_{G}$ is the celebrated \emph{Dixmier-Macaev ideal},
 \begin{equation*}
 \sM_{1,\infty}:=\bigg\{ T\in \sK; \ \sum_{j<N} \mu_j(T) =\op{O}(\log N)\bigg\}.
\end{equation*}
\end{example}

\begin{remark}
If $\rho=-1$ and $\int_0^\infty g(t)dt=\infty$ (e.g., $g(t)=(t+1)^{-1}(\log (t+2))^{q}$ with $q>-1$), then by Karamata's theorem
(Theorem~\ref{thm:karamata}) we have $tg(t)=\op{o}(G(t))$. It follows that
\begin{equation*}
 \sL_g\subseteq \sV_g, \qquad \text{where}\ \sV_g:=\left\{T\in \sM_{G};\ j\mu_j(T)=\op{o}(G(j))\right\}.
\end{equation*}
As $\sV_g\subsetneq \sM_{G}$ (\emph{cf}.\ \cite[Proposition~10]{SS:JFA13}), we see that  $\sL_g\subsetneq \sM_{G}$. In fact, the operator considered in the proof of~\cite[Proposition~10]{SS:JFA13} is an example of an operator in $\sM_{G}$ which is not in $\sL_g$ (see also~\cite[Lemma~1.2.8]{LSZ:Book}).
\end{remark}

\section{Dixmier Traces and Connes' Integration on Weak Lorentz Ideals} \label{sec:Dixmier}
In this section, we extend to weak Lorentz ideals the approach of~\cite{LSZ:Book, Po:JNCG23, Su:CIMPA17} to Dixmier traces and to Connes' integral on the weak trace class $\sL_{1,\infty}$. This simplifies previous approaches to Dixmier traces on Lorentz ideals and provides full analogues for weak Lorentz ideals of the results of~\cite{LSZ:Book, Po:JNCG23}.

Throughout this section, we let $g:[0,\infty)\rightarrow (0,\infty)$ be a continuous $\RV_{-1}$ function such that $\int_0^\infty g(t)dt=\infty$. We also set $G(t)=\int_0^t g(s)ds$. These assumptions ensure that
\begin{equation}\label{eq:Lorentz.divergence}
 \sum_{j<N} g(j)= G(N) +\op{O}(1) \qquad \text{and} \qquad \sum_{j\geq 0} g(j) =\infty.
\end{equation}
In particular, this implies that $\sL_1\subsetneq \sL_g$.

In addition, by Karamata's theorem (Theorem~\ref{thm:karamata}) we have
\begin{equation}\label{eq:Karamata-RV1}
 tg(t) =\op{o}(G(t)) \qquad \text{as $t\rightarrow \infty$}.
\end{equation}

\subsection{Traces on $\sL_g$}
By definition, the \emph{commutator space} of $\sL_g$ is
\begin{equation*}
 \Com(\sL_g):= \Span\left\{[A,T]; \ A \in \sL(\sH), \ T\in \sL_g\right\}.
\end{equation*}
Using the fact that every operator $A\in \sL(\sH)$ is a linear combination of 4 unitary operators (see, e.g., \cite[\S VI.6]{RS1:1980}), it can be shown that
\begin{equation}\label{def:com}
 \Com(\sL_g):= \Span\left\{U^*TU-T;  \ T\in \sL_g,\ U \in \sL(\sH), \ U\ \text{unitary}\right\}.
\end{equation}

We shall say that a linear functional $\varphi: \sL_g\rightarrow \C$ is a \emph{trace} if it annihilates the commutator space $ \Com(\sL_g)$. That is,
\begin{equation*}
 \varphi(AT)=\varphi(TA) \qquad \forall T\in \sL_g \ \forall A\in \sL(\sH).
\end{equation*}
 In fact, in view of~(\ref{def:com}), the above trace property is equivalent to being \emph{unitarily invariant}, i.e.,
 \begin{equation}\label{eq:uni invariance}
 \varphi(U^*TU)=\varphi(T) \qquad T\in \sL_g \ \forall U\in \sL(\sH), \ U\ \text{unitary}.
\end{equation}

More generally, given any subspace $\sE\subseteq \sL_g$ containing $\Com(\sL_g)$ we shall call a trace any linear functional on $\sE$ that annihilates
$\Com(\sL_g)$.

We also recall that a linear functional $\varphi: \sL_g\rightarrow \C$ is called \emph{positive} if
\begin{equation*}
 \left( T\in \sL_g \quad \text{and} \quad T\geq 0\right) \ \Longrightarrow \ \varphi(T)\geq0.
\end{equation*}

\subsection{Inequalities for eigenvalues}
For an arbitrary compact operator $T$, the non-zero part of the spectrum of $T$ consists of isolated eigenvalues. For each $\lambda \in \Sp(T)\setminus \{0\}$, the associated root space is $E_\lambda(T):= \cup_{j\geq 0} \ker (T-\lambda)^j$. This is a finite-dimensional space and its dimension is called the \emph{algebraic multiplicity} of $\lambda$.

\begin{definition}
 An \emph{eigenvalue sequence} for $T$ is any sequence $\lambda(T)=(\lambda_j(T))_{j\geq 0}$ consisting of all eigenvalues of $T$ in such a way that
 \begin{itemize}
 \item Each eigenvalue is repeated according to its algebraic multiplicity.

 \item We have
 \begin{equation*}
 |\lambda_0(T)|\geq |\lambda_1(T)|\geq  |\lambda_2(T)|\geq \cdots \geq 0.
\end{equation*}
\end{itemize}
 \end{definition}

\begin{remark}
If $T\geq 0$, then $\mu(T)=(\mu_j(T))_{j\geq 0}$ is the unique eigenvalue sequence of $T$.
\end{remark}
\begin{remark}
 In general, an eigenvalue sequence need not be unique.
\end{remark}

It is immediate that if $\lambda(T)$ is an eigenvalue sequence for $T$, then $c\lambda(T)$ is an eigenvalue sequence for $cT$ for all $c\in \C$. In addition, we have the following Weyl inequality (see, e.g., \cite{GK:AMS69, Si:AMS05}),
\begin{equation}\label{eq:NC-integration.Weyl-inequality}
 \sum_{j<N} |\lambda_j(T)| \leq \sum_{j<N} \mu_j(T).
\end{equation}

In what follows, given any operator $T\in \sK$, by $\lambda(T)$ we shall mean an arbitrary eigenvalue sequence for $T$.

As mentioned in the Introduction, the approach of~\cite{LSZ:Book, Po:JNCG23, Su:CIMPA17} to Dixmier traces on the weak trace class $\sL_{1,\infty}$ relies on the asymptotic additivity property~(\ref{eq:Intro.asymptotic-additivity-lunf}).  The key observation in this section is that we have an analogous result for operators in $\sL_g$.

For one thing, it was shown by Kalton~\cite[Proposition~3.2]{Ka:Crelle98} that any quasi-Banach ideal is geometrically stable, and hence is closed under logarithmic submajorization in the sense of~\cite{SZ:AIM14} (see~\cite[Lemma~35]{SZ:AIM14}; see also~\cite[Proposition~2.4.18]{LSZ:Book}). In particular, the ideal $\sL_g$ is closed under logarithmic submajorization. We may thus apply~\cite[Lemma~5.6.4]{LSZ:Book} to get that, for all $S,T\in \sL_g$, we have
  \begin{equation}\label{eq:1st asymptotic}
  \sum_{j<N} \lambda_j(S+T)= \sum_{j<N} \lambda_j(S)+ \sum_{j<N} \lambda_j(T) + \op{O}\left(Ng(N)\right).
\end{equation}
On the other hand, by the Karamata asymptotic~(\ref{eq:asymptotic--karamata}) we have $Ng(N)=\op{o}(G(N))$. Therefore, we arrive at the following result.

\begin{proposition}[Asymptotic Additivity]\label{prop:asymptotic-additivity}
Given any operators $S,T\in \sL_g$, we have
 \begin{equation*}
 \sum_{j<N} \lambda_j(S+T)= \sum_{j<N} \lambda_j(S)+ \sum_{j<N} \lambda_j(T) + \op{o}\left(G(N)\right).
\end{equation*}
\end{proposition}

We mention a few direct consequences of Proposition~\ref{prop:asymptotic-additivity}.

\begin{corollary}\label{cor:comparison}
 If $T\in \sL_g$, and $\lambda(T)$ and $\lambda'(T)$ are eigenvalue sequences for $T$, then
 \begin{equation*}
  \sum_{j<N} \lambda_j'(T) =\sum_{j<N} \lambda_j(T) + \op{o}\left(G(N)\right).
\end{equation*}
\end{corollary}
\begin{proof}
 This follows from Proposition~\ref{prop:asymptotic-additivity} for $S=0$ and $\lambda(S+T)=\lambda'(T)$.
\end{proof}

\begin{corollary}[Kalton~\cite{Ka:Crelle98}]\label{cor:com--spectral}
 If $T\in \Com(\sL_g)$, then
\begin{equation}\label{eq:spectral2}
 \sum_{j<N} \lambda_j(T)= \op{o}\left(G(N)\right).
\end{equation}
\end{corollary}
\begin{proof}
 In view of~(\ref{def:com}) and Proposition~\ref{prop:asymptotic-additivity} it is enough to prove the result for operators of the form,
 \begin{equation*}
 U^*TU-T, \qquad T\in \sL_g, \quad U\in \sL(\sH),\ U\ \text{unitary}.
\end{equation*}
We just need to apply Proposition~\ref{prop:asymptotic-additivity} to the pair $(U^*TU, -T)$ with $\lambda(U^*TU)=
\lambda(T)=-\lambda(-T)$. This is possible, since $U^*TU$ and $T$ have the same eigenvalues with the same multiplicities. We thus get
\begin{equation*}
  \sum_{j<N} \lambda_j(U^*TU-T)= \sum_{j<N} \lambda_j(T)+\sum_{j<N} \left(-\lambda_j(T) \right)+ \op{o}\left(G(N)\right)= \op{o}\left(G(N)\right).
\end{equation*}
This gives the result.
\end{proof}

\begin{corollary}\label{cor:separable part--spectral}
If $T\in (\sL_{g})_0$, then
  \begin{equation*}
 \sum_{j<N} \lambda_j(T)= \op{o}\left(G(N)\right).
\end{equation*}
\end{corollary}
\begin{proof}
As $T \in (\sL_{g})_0$, we have $\mu_j(T)=\op{o}(g(j))$. Combining this with~(\ref{eq:NC-integration.Weyl-inequality}) and~(\ref{eq:spectral2}) gives the result.
\end{proof}

Corollary~\ref{cor:com--spectral} is actually a special case of a much more general result, which we mention for the sake of completeness.

\begin{proposition}[{\cite[Theorem~5.1.5]{LSZ:Book}}; see also~\cite{Ka:Crelle98}]\label{prop:Lorentz.Com-characterization}
If $S,T\in \sL_g$, then
\begin{equation}\label{eq:com difference}
S-T \in \Com\left(\sL_g\right) \ \Longleftrightarrow \  \sum_{j<N} \lambda_j(S)=\sum_{j<N} \lambda_j(T)+ \op{O}(Ng(N)).
\end{equation}
In particular, we have
 \begin{equation}\label{eq:spectral1}
 T \in \Com\left(\sL_g\right) \ \Longleftrightarrow \ \sum_{j<N} \lambda_j(T)= \op{O}(Ng(N)).
\end{equation}
\end{proposition}

\begin{remark}
 The above result is proved in~\cite{LSZ:Book} for ideals that are closed under logarithmic sub-majorization. The spectral characterization of commutators provided by~(\ref{eq:spectral1}) was originally proved by Kalton~\cite{Ka:Crelle98} for geometrically stable ideals.
 We also note that we can get~(\ref{eq:com difference}) from~(\ref{eq:spectral1}) by using~(\ref{eq:1st asymptotic}).
\end{remark}

\begin{remark}
 The above result was proved in~\cite{DK:Crelle98} in the general setting of geometrically stable ideals. Every quasi-Banach ideal is geometrically stable (see~\cite{Ka:Crelle98}). This result is further extended to ideals closed under logarithmic submajorization in~\cite[Theorem~5.6.1]{LSZ:Book}.
\end{remark}

\subsection{Extended limits}
In what follows we let $\ell_\infty=\ell_\infty(\N_0)$ be the unital $C^*$-algebra of bounded complex-valued sequences $a=(a_j)_{j\geq 0}$ with norm,
\begin{equation*}
 \|a\|_\infty = \sup_{j\geq 0} |a_j|, \qquad a=(a_j)_{j\geq 0} \in \ell_\infty.
\end{equation*}
Let $\co$ be the (closed) ideal of sequences converging to zero. The quotient algebra $\ell_\infty/\co$ is a $C^*$-algebra with respect to the quotient norm and the induced involution. Moreover, the quotient map $\pi:\ell_\infty\rightarrow \ell_\infty/\co$ is a contracting $*$-homomorphism.

\begin{definition}
 An \emph{extended limit} is a positive linear functional $\omega:\ell_\infty\rightarrow \C$ such that:
 \begin{itemize}
 \item[(i)] $\omega(1) = 1$, where $1$ is the constant sequence whose entries are all equal to $1$.

 \item[(ii)] $\omega(a)=0$ for all sequences $a=(a_j)_{j\geq 0} \in \co$.
\end{itemize}
We denote by $\EL(\N_0)$ the set of extended limits.
 \end{definition}

\begin{remark}
 Together (i)--(ii) mean that
\begin{equation}\label{eq:true to extended}
 \lim_{j\rightarrow \infty} a_j = L \ \Longrightarrow \ \omega(a)=L.
\end{equation}
Thus, an extended limit is precisely a \emph{positive} extension of the limit.
\end{remark}

\begin{remark}
 If $\omega\in \EL(\N_0)$, then its positivity has some consequences. Given any $a=(a_j)_{j\geq 0}\in \ell_\infty$, we have
 \begin{equation*}
 \omega( \overline{a})= \overline{\omega(a)}, \qquad \Re(\omega(a))= \omega \left(\Re a\right), \qquad \Im(\omega(a))= \omega \left(\Im a\right).
\end{equation*}
Moreover, together with the fact that $\omega(1)=1$ positivity implies that
\begin{equation*}
 |\omega(a) |\leq \omega\left( |a|\right) \leq \omega(\|a\|_\infty \cdot 1)=\|a\|_\infty.
\end{equation*}
Thus, $\omega$ is a continuous linear form on $\ell_\infty$ with $\|\omega\|=1$. Conversely, any continuous linear form $\ell:\ell_\infty\rightarrow \C$ such that $\ell(1)=1$ and $\|\ell\|=1$ must be positive (see, e.g., \cite[pp.\ 83--84]{Co:Springer10}).
\end{remark}

\begin{remark}
 If $\omega: \ell_\infty\rightarrow \C$ is an extended limit, then its positivity and (i)--(ii) ensure that $\omega$ descends to a (positive) state on the unital $C^*$-algebra $\ell_\infty/\co$. Conversely, any state $\omega:\ell_\infty/\co\rightarrow \C$ uniquely lifts to a positive linear form on $\ell_\infty$ satisfying (i)--(ii). We thus have a natural one-to-one correspondence between $\EL(\N_0)$ and the states on $\ell_\infty/\co$.
\end{remark}

\begin{lemma}\label{lem:el property}
Let $a=(a_j)\in \ell_\infty$, and set  $\EL(a):= \left\{ \omega(a); \ \omega \in \EL(\N_0)\right\}$.
\begin{enumerate}
 \item[(i)] $\EL(a)$ is a convex subset of $\C$ containing all the cluster points of $a$.

 \item[(ii)] If $a$ is real-valued, then
 \begin{equation*}
 \EL(a) = \left[\liminf a, \limsup a\right].
\end{equation*}

\item[(iii)] We have
\begin{equation*}
 \lim_{j\rightarrow \infty}a_j = L \ \Longleftrightarrow \ \big( \omega(a)=L \quad \forall \omega \in \EL(\N_0)\big).
\end{equation*}
\end{enumerate}
\end{lemma}
\begin{proof}
 The convexity of $\EL(a)$ is a direct consequence of the convexity of $\EL(\N_0)$. Suppose that $\alpha$ is a cluster point of $a$. Let $(a_{k_j})_{j\geq 0}$ be a subsequence converging to $\alpha$. Given $\omega\in \EL(\N_0)$, we define another extended limit $\tilde{\omega}$ by
 \begin{equation*}
 \tilde{\omega}(b):=\omega\left(\{b_{k_j}\}\right), \qquad b=(b_j)\in \ell_\infty.
\end{equation*}
We then have $\tilde{\omega}(a)=\omega(\{a_{k_j}\})=\alpha$, and hence $\alpha \in \EL(a)$. This proves~(i).

Assume that $a$ is real-valued.  As $\liminf a$ and  $\limsup a$ are cluster points of $a$,  part~(i) implies that $[\liminf a, \limsup a]\subseteq \EL(a)$. Moreover, we
have the obvious inequalities,
\begin{equation*}
 \inf_{k\geq j}a_k \leq a_j \leq \sup_{k\geq j} a_k \qquad \forall j\geq 0.
\end{equation*}
Therefore, if $\omega\in \EL(\N_0)$, then its positivity and~(\ref{eq:true to extended}) imply
\begin{equation*}
 \liminf_{j\rightarrow \infty} a_j = \lim_{j\rightarrow\infty} \big(\inf_{k\geq j}a_k \leq a_j\big)\leq  \omega(a) \leq  \lim_{j\rightarrow\infty} \big(\sup_{k\geq j}a_k\big)
 = \limsup_{j\rightarrow \infty} a_j.
\end{equation*}
This proves part~(ii).

Part (iii) follows from (ii) and~(\ref{eq:true to extended}). The proof is complete.
\end{proof}

\subsection{Dixmier traces}
If $T\in \sL_g$, then, as in~(\ref{eq:Lorentz norm}), Weyl's inequality~(\ref{eq:NC-integration.Weyl-inequality}) implies that
\begin{equation}\label{eq:bdns partial sum}
  \bigg| \sum_{j<N} \lambda_j(T) \bigg| \leq \sum_{j<N} |\lambda_j(T)| \leq \sum_{j<N} \mu_j(T) \leq G(N)\|T\|_G .
\end{equation}
In particular, we see that
\begin{equation*}
\bigg\{ \frac{1}{G(N)} \sum_{j<N} \lambda_j(T)\bigg\}_{N\geq 0} \in \ell_\infty.
\end{equation*}
(By convention $\sum_{j<0} \lambda_j(T)=0$.) Moreover, it follows from Corollary~\ref{cor:comparison} that the class
$\pi(\lambda(T))\in \ell_\infty/\co$ does not depend on the choice of the eigenvalue sequence $\lambda(T)$ for $T$. It follows that, given any extended limit $\omega$, we have a well defined functional $\Trw: \sL_g \rightarrow \C$ such that
\begin{equation}\label{eq:def of dixmier}
 \Trw(T):= \omega \bigg( \bigg\{\frac{1}{G(N)} \sum_{j<N} \lambda_j(T)\bigg\}_{N\geq0}\bigg),
\end{equation}
where $\lambda(T)$ is any eigenvalue sequence for $T$.

\begin{proposition}\label{prop:dixmier}
$\Trw$ is a normalized positive trace on $\sL_g$ annihilating $(\sL_g)_0$. Moreover, we have
\begin{equation}\label{eq:cont-dixmier}
 \left| \Trw(T)\right| \leq \Trw\left(|T|\right) \leq \|T\|_{G} \qquad \forall T\in \sL_g.
\end{equation}
\end{proposition}
\begin{proof}
It is immediate from Proposition~\ref{prop:asymptotic-additivity} that $\Trw$ is additive. If $c\in \C$, then $c\lambda(T)$ is an eigenvalue sequence for $\lambda(cT)$, and so we have
\begin{equation*}
 \Trw(cT)=\omega\bigg(c\bigg\{G(N)^{-1}\sum_{j<N} \lambda_j(T)\bigg\}\bigg)=c\Trw(T).
\end{equation*}
It also follows from Corollary~\ref{cor:com--spectral} and Corollary~\ref{cor:separable part--spectral} that $\Trw$ annihilates $\Com(\sL_g)\cup (\sL_g)_0$. In addition, if $T\geq 0$, then $\lambda(T)=\mu(T)$, and so we have
\begin{equation*}
 \Trw(T)= \omega \bigg( \bigg\{\frac{1}{G(N)} \sum_{j<N} \mu_j(T)\bigg\}\bigg)\geq 0.
\end{equation*}
Therefore, $\Trw$ is a positive trace on $\sL_g$ that annihilates $(\sL_g)_0$. The inequalities~(\ref{eq:cont-dixmier}) follow from~(\ref{eq:bdns partial sum}) and~(\ref{eq:def of dixmier}). The proof is complete.
\end{proof}

\begin{remark}
 The inequalities~(\ref{eq:cont-dixmier}) and the continuity of the inclusion of $\sL_g$ into $\sM_{G}$ implies that $\Trw$ is a continuous trace. In fact, every positive trace on $\sL_g$ is continuous (see Proposition~\ref{prop:cont--positive}).
\end{remark}

\begin{remark}
 The fact that $\Trw$ annihilates $(\sL_g)_0$ implies it is a singular trace, i.e., it annihilates finite-rank operators. It is actually a general fact that if $\sJ$ is a Banach ideal that does not contain $\sL_1$, then every continuous trace on $\sJ$ is singular (see~\cite[Theorem~3.5.8]{LSZ:Book}).
\end{remark}

\begin{definition}
 $\Trw$ is called the \emph{Dixmier trace} associated with the extended limit $\omega$.
\end{definition}

\begin{remark}\label{rmk:value--dixmier}
Let $T\in \sL_g$, and set $\op{DT}(T):=\{\Trw(T); \ \omega\in \EL(\N_0)\}$. Given any eigenvalue sequence $\lambda(T)$,
Lemma~\ref{lem:el property} implies the following:
\begin{itemize}
 \item $\op{DT}(T)$ is a convex set containing all the cluster points of the sequence $\{\frac{1}{ G(N)} \sum_{j<N}\lambda_j(T)\}$.

 \item If $T$ is selfadjoint, then
 \begin{equation*}
 \op{DT}(T)= \bigg[ \liminf_{N\rightarrow \infty}   \frac{1}{G(N)} \sum_{j<N} \lambda_j(T),  \limsup_{N\rightarrow \infty}   \frac{1}{G(N)} \sum_{j<N} \lambda_j(T) \bigg].
\end{equation*}
\end{itemize}
\end{remark}

\begin{remark}
Dixmier~\cite{Di:CRAS66} originally defined his traces on the larger Banach ideal $\sM_{G}$ in terms of extended limits that are invariant under the dilation $(a_j)_{j\geq 0} \rightarrow (a_{[j/2]})_{j\geq 0}$. Dixmier observed that, for any dilation-invariant extended limit $\omega$, we define an additive functional on the positive cone of $\sM_{G}$ by letting
\begin{equation*}
 \Trw(T) = \omega \bigg( \bigg\{\frac{1}{G(N)} \sum_{j<N} \mu_j(T)\bigg\}\bigg), \qquad T \in (\sM_{G})_+.
\end{equation*}
Therefore, it uniquely extends to a positive trace on $\sM_{G}$, which is called a Dixmier trace as well.

Note that Proposition~\ref{prop:asymptotic-additivity} does not hold for operators in $\sM_{G}$. In fact, even if the formula~(\ref{eq:def of dixmier}) makes sense for operators in $\sM_G$, this need not define an additive functional on $\sM_G$ (see~\cite[Example~6.6.1]{LSZ:Book}). However, as shown by Sedaev-Sukochev~\cite[Corollary~18]{SS:JFA13}, given any extended limit $\omega$,  there is always a dilation-invariant extended limit $\omega'$ such that $\Trw$ agrees on $\sL_g$ with the Dixmier trace $\Tr_{\omega'}$ (which is defined on $\sM_{G}$). Thus, every Dixmier trace on $\sL_g$ in the sense of~(\ref{eq:def of dixmier}) is the restriction of a Dixmier trace on $\sM_{G}$ in the sense of Dixmier.
\end{remark}

\subsection{Measurable operators}

\begin{definition}
 An operator $T\in \sL_g$ is called \emph{measurable} (or \emph{Dixmier-measurable}) if all Dixmier traces on $\sL_g$ take the same value on $T$.
 \end{definition}

We denote by $\sM(\sL_g)$ the set of measurable operators in $\sL_g$.

\begin{definition}
 If $T\in \sL_g$ is measurable, then its \emph{noncommutative integral} is defined by
 \begin{equation*}
 \gbint T := \Trw(T),
\end{equation*}
 where $\Trw$ is any Dixmier trace on $\sL_g$.
 \end{definition}

\begin{remark}
 In the case of the weak trace class $\sL_{1,\infty}$, i.e., $g(t)=(t+1)^{-1}$, we shall denote the NC integral just by $\bint$, since this is the usual notation in this case.
\end{remark}

 \begin{proposition}\label{prop:measurable-space-properties}
 The following hold.
 \begin{enumerate}
 \item  $\sM(\sL_g)$ is a closed subspace of $\sL_g$ containing $\overline{\Com(\sL_g)}$. In particular, $\Com(\sL_g)$ and $(\sL_g)_0$ are subspaces of
 $\sM(\sL_g)$.

 \item The NC integral $\gbint: \sM(\sL_g) \rightarrow \C$ is a continuous positive linear trace that annihilates $(\sL_g)_0$.
\end{enumerate}
\end{proposition}
\begin{proof}
 By definition,
 \begin{equation*}
 \sM(\sL_g) =\bigcap_{\omega,\omega'} \big\{T\in \sL_g;\ \Tr_\omega(T)=\Tr_{\omega'}(T)\big\},
\end{equation*}
where $\omega$ and $\omega'$ range over all states on $\ell_\infty/\co$. As the Dixmier traces $\Tr_\omega$ are continuous linear maps, it follows that $ \sM(\sL_g)$ is a closed subspace of $\sL_g$. This gives the first part. The second part follows from Proposition~\ref{prop:dixmier} and the fact that $\gbint$ agrees with any Dixmier trace on $ \sM(\sL_g)$.
\end{proof}

\subsection{Spectral characterization}
For the weak trace class $\sL_{1,\infty}$, the equivalence~(\ref{eq:Intro.spectral-measurability-lunf}) provides a spectral characterization of measurable operators and a spectral interpretation of the NC integral (see~\cite{Co:NCG, LSZ:Book, Po:JNCG23}). The following statement extends this result to all weak Lorentz ideals.

\begin{theorem}\label{thm:measurable}
 Given any $T\in \sL_g$, the following are equivalent:
 \begin{enumerate}
 \item[(i)] $T$ is measurable and $\gbint T=L$.

 \item[(ii)]  ${\displaystyle \lim_{N\rightarrow \infty}G(N)^{-1} \sum_{j < N} \lambda_j(T)= L}$ for \underline{some} eigenvalue sequence $\lambda(T)$.

 \item[(iii)] ${\displaystyle \lim_{N\rightarrow \infty}G(N)^{-1} \sum_{j < N} \lambda_j(T)= L}$ for \underline{every} eigenvalue sequence $\lambda(T)$.
\end{enumerate}
\end{theorem}
\begin{proof}
 By Corollary~\ref{cor:comparison}, if $\lambda(T)$ and $\lambda'(T)$ are eigenvalue sequences for $T$, then
 \begin{equation*}
 \frac{1}{G(N)} \sum_{j<N}\lambda_j'(T) = \frac{1}{G(N)} \sum_{j<N}\lambda_j(T)+\op{o}(1).
\end{equation*}
 Thus, if the sequence  $\big(G(N)^{-1} \sum_{j<N}\lambda_j(T)\big)_{N\ge0}$ has a limit for \emph{some} eigenvalue sequence, then it has a limit for \emph{every} eigenvalue sequence, and the value of the limit does not depend on the choice of the eigenvalue sequence. This gives the equivalence of (ii) and (iii).

 Moreover, if $\lambda(T)$ is any eigenvalue sequence for $T$, then it follows from Remark~\ref{rmk:value--dixmier} that
\begin{align*}
  \lim_{N\rightarrow \infty} \frac{1}{G(N)} \sum_{j<N}\lambda_j(T) =L
  \; \Longleftrightarrow  &  \ \omega \bigg(\bigg\{\frac{1}{G(N)} \sum_{j<N}\lambda_j(T) \bigg\}\bigg) =L \quad \forall \omega \in \EL(\N_0)\\
 \Longleftrightarrow   & \  \Trw(T)=L \quad \forall \omega\\
 \Longleftrightarrow   & \ T\ \text{is measurable and}\ \gbint T=L.
\end{align*}
It follows from this that (ii)$\Rightarrow$(i) and (i)$\Rightarrow$(iii), and so (i), (ii), and (iii) are equivalent. The proof is complete.
\end{proof}

\begin{corollary}\label{cor:Dixmier.independence-g}
 Let $g_1(t)$, $t\geq 0$, be any $\RV_{-1}$-function such that $g_1(t)\sim g(t)$. Then the spaces $\sM(\sL_{g_1})$ and $\sM(\sL_{g})$ agree, and on this space $\gubint=\gbint$.
\end{corollary}
\begin{proof}
 Set $G_1(t)=\int_0^t g_1(s)ds$. As $g_1(t)\sim g(t)$ and $\int_0^\infty g(t)dt=\infty$, we see that $G_1(t)\sim G(t)$. Thus, if $T\in \sL_g$, then
 \begin{equation*}
 \bigg(  \lim_{N\rightarrow \infty} \frac{1}{G_1(N)} \sum_{j<N}\lambda_j(T) =L\bigg) \Longleftrightarrow  \bigg(  \lim_{N\rightarrow \infty} \frac{1}{G(N)} \sum_{j<N}\lambda_j(T) =L\bigg).
\end{equation*}
 Combining this with Theorem~\ref{thm:measurable} gives the result.
\end{proof}

Let $\sH'$ be another Hilbert space. In order to distinguish them we shall denote by $\sL_g(\sH)$ and $\sL_g(\sH')$ the quasi-Banach ideals of $\sL_g$-operators on $\sH$ and $\sH'$, respectively.

In what follows, we shall say that two operators $S\in \sL_g(\sH)$ and $T\in \sL_g(\sH')$ have the same \emph{non-zero spectrum up to multiplicity} if they have the same non-zero eigenvalues with the same (algebraic) multiplicities. Equivalently, any eigenvalue sequence for one operator is an eigenvalue sequence for the other operator.

The very fact that Theorem~\ref{thm:measurable} provides us with a characterization of measurability and $\gbint$ in terms of eigenvalue sequences immediately implies a strong form of spectral invariance. Namely, we have the following result.

\begin{corollary}\label{cor:Meas.spectral-invariance}
 If  $S\in \sL_g(\sH)$ and $T\in \sL_g(\sH')$ have the same non-zero spectrum up to multiplicity, then $S$ is measurable if and only if $T$ is measurable. Moreover, in this case $\gbint S=\gbint T$.
\end{corollary}

In the special case where $\sH'$ is the Hilbert space $\sH$ with an equivalent inner product, we obtain the following result.

\begin{corollary}
 $\sM(\sL_g(\sH))$ and $\gbint$ do not depend on the inner product of $\sH$.
\end{corollary}

\subsection{Connes' question}\label{subsec:Connes-question}
 From the sole perspective of spectral theory, Theorem~\ref{thm:measurable} implies the following result.

\begin{proposition}\label{prop:tauberian}
 The following formula,
 \begin{equation}\label{eq:nci tauberian}
 {\gbint}' T:= \lim_{N\rightarrow \infty} \frac{1}{G(N)}\sum_{j < N} \lambda_j(T) \quad \text{whenever the limit exists},
\end{equation}
defines a positive continuous linear form with the following properties:
\begin{enumerate}
 \item[(i)] Its domain is a closed subspace of $\sL_g$ containing the commutator subspace $\Com(\sL_g)$ and the separable ideal $(\sL_g)_0$.

  \item[(ii)] It is positive and continuous on its domain.

\item[(iii)]  It annihilates both $\Com(\sL_g)$ and $(\sL_g)_0$, and so this is a singular trace on its domain.
 \end{enumerate}
\end{proposition}

In the case of the weak trace class $\sL_{1,\infty}$, Connes asked for a direct proof of Proposition~\ref{prop:tauberian} that does not involve any use of, or reference to, extended limits (see~\cite[Question~A]{Po:JNCG23}). Extended limits are in one-to-one correspondence with states on the $C^*$-algebra $\ell_\infty/\co$. Their existence is then a consequence of applying the Hahn-Banach theorem on $\ell_\infty/\co$. As $\ell_\infty/\co$ is non-separable, this requires assuming the Axiom of Choice. Therefore, one motivation of Connes was to obtain a construction of the NC integral that does not rely on the Axiom of Choice.

A positive answer to Connes' question was provided in~\cite{Po:JNCG23}. This was a consequence of a version of Proposition~\ref{prop:asymptotic-additivity} for the weak trace class $\sL_{1,\infty}$ that was established in the first edition of~\cite{LSZ:Book}.

As we shall now proceed to show, the full version of Proposition~\ref{prop:asymptotic-additivity} enables us to get a direct proof of Proposition~\ref{prop:tauberian}. This will extend the scope of applicability of the answer to Connes' question to all weak Lorentz ideals.

\begin{proof}[Direct proof of Proposition~\ref{prop:tauberian}]
The first part of the proof of Theorem~\ref{thm:measurable} shows that Proposition~\ref{prop:asymptotic-additivity} implies the equivalence of  (ii) and (iii). Note that the arguments there do not make any reference to extended limits. Therefore, the domain of ${ \gbint}'$ is
\begin{equation*}
 \sM_T(\sL_g):= \biggl\{ T\in \sL_g;\ \lim_{N\rightarrow \infty} \frac{1}{G(N)} \sum_{j<N}\lambda_j(T)\ \textup{exists} \biggr\}.
\end{equation*}
In the terminology of~\cite{LSZ:Book} the operators in $\sM^g_T$ are called Tauberian operators.

 The fact that $  \sM_T(\sL_g)$ is a subspace of $\sL_g$ and $\gbint'$ is a linear functional is a direct consequence of Proposition~\ref{prop:asymptotic-additivity}. It is then immediate from its definition that $\gbint'$ is a positive linear functional. It follows from Corollary~\ref{cor:com--spectral} and Corollary~\ref{cor:separable part--spectral} that if $T\in (\sL_g)_0\cup \Com(\sL_g)$, then the r.h.s.\ of~(\ref{eq:nci tauberian}) vanishes. This means that $\Com(\sL_g)\cup(\sL_g)_0$ is contained in $\sM_T(\sL_g)$ and annihilated by $\gbint'$. In particular, $\gbint'$ is a singular trace on $ \sM_T(\sL_g)$.

 It remains to show that $ \sM_T(\sL_g)$ is closed and $\gbint'$ is continuous. For $T\in \sL_g$, set
 \begin{equation*}
 \tau_N(T):= \frac{1}{G(N)} \sum_{j<N}\lambda_j(T), \qquad N\geq 1.
\end{equation*}
It follows from Proposition~\ref{prop:asymptotic-additivity} and the inequalities~(\ref{eq:bdns partial sum}) that, for all $S,T\in \sL_g$,  we have
\begin{align*}
 \left| \tau_N(S)-\tau_N(T)\right| & \leq |\tau_N(S-T)| +\op{o}(1)\\ & \leq \|S-T\|_G +\op{o}(1).
\end{align*}
Using these inequalities it can be shown that, if $T\in \sL_g$ and $(T_\ell)_{\ell\geq 0}\subseteq  \sM^g_T$ converges to $T$ in $\sL_g$, then the following hold:
\begin{enumerate}
 \item[(a)] $\bigg\{\gbint' T_\ell\bigg\}_{\ell\geq 0}$ is a Cauchy sequence, and hence has a limit.

 \item[(b)] We have
\begin{equation*}
 \lim_{N\rightarrow \infty} \tau_N(T)= \lim_{\ell \rightarrow \infty} {\gbint}' T_\ell.
\end{equation*}
\end{enumerate}
That is, $T\in  \sM_T(\sL_g)$ and $\gbint' T=\lim_{\ell \rightarrow \infty} \gbint' T_\ell$. This shows that $ \sM_T(\sL_g)$ is closed and $\gbint'$ is continuous. This completes the direct proof of Proposition~\ref{prop:tauberian}.
\end{proof}

\section{Birman-Solomyak Perturbation Theory for Weak Lorentz Ideals}\label{sec:Bir-Sol}
In this section, we extend to weak Lorentz ideals the perturbation theory of Birman-Solomyak for eigenvalues and singular values of operators in weak-Schatten classes $\sL_{p,\infty}$, $p>0$~(see, e.g., \cite[\S4]{BS:JFAA70} and~\cite[\S11.6]{BS:Book}). The main observation is that the $\RV$-property is the very property that allows us to extend the original approach of Birman-Solomyak to weak Lorentz ideals.

Throughout this section we let $g:[0,\infty)\rightarrow (0,\infty)$ be a continuous $\RV_\rho$ function with $\rho<0$.

\subsection{Positive and negative eigenvalues of compact operators}
If $T$ is a selfadjoint compact operator, then its spectrum consists of real eigenvalues with finite multiplicity. We then denote by $(\pm \lambda_j^\pm(T))$ its sequences of positive and negative eigenvalues ordered in such a way that
\begin{equation*}
 \lambda_0^\pm(T)\geq \lambda_1^\pm(T) \geq \lambda_2^\pm(T)\geq \cdots \geq 0,
\end{equation*}
where each eigenvalue is repeated according to multiplicity. Equivalently, we have
\begin{equation*}
 \lambda_j^\pm(T)=\mu_j(T^\pm), \qquad j\geq 0,
\end{equation*}
where $T^\pm$ are the positive and negative parts of $T$, i.e.,
\begin{equation*}
 T^+=\frac{1}{2} \big(|T|+T), \qquad T^-=\frac{1}{2} \big(|T|-T).
\end{equation*}
Note that
\begin{equation*}
  \lambda_j^\pm(cT)=c\lambda_j^\pm(T), \quad c>0, \qquad  \lambda_j^\pm(-T)=\lambda_j^\mp(T).
\end{equation*}

We have the following version of the min-max principle for positive and negative eigenvalues (see, e.g.,~\cite[Theorem~9.2.4]{BS:Book}):
\begin{equation*}
 \lambda_j^\pm(T) = \min_{\dim E=j} \max_{\xi\in E^\perp \setminus \{0\}} \pm \frac{\scal{T\xi}{\xi}}{\scal{\xi}{\xi}}.
\end{equation*}
Comparing this to the min-max principle~(\ref{eq:min-max}) shows that
\begin{equation}\label{eq:domi1}
 0\leq \lambda_j^\pm(T) \leq \mu_j(T).
\end{equation}
In addition, we have the following Weyl inequality (see, e.g., \cite[Theorem~9.2.8]{BS:Book}),
\begin{equation}\label{eq:Bir-Sol.Weyl-Ineq}
 \lambda_{j+k}^\pm(S+T)\leq  \lambda_{j}^\pm(S)+  \lambda_{k}^\pm(T).
\end{equation}

\subsection{Quasi-norm topology on $\sL_g/(\sL_g)_0$}
Given any $T\in \sL_g$, we shall denote by $\dot{T}$ its class in the quotient space $\dot{\sL}_g:=\sL_g/(\sL_g)_0$. We equip
$\dot{\sL}_g$ with the topology of the quasi-norm defined by
\begin{equation*}
\|\dot{T}\|_{\dot{g}}:= \inf\big\{\|T+S\|_g;\ S\in (\sL_g)_0\big \}, \qquad T\in \sL_g.
\end{equation*}
By definition, $(\sL_g)_0$ is the closure of the ideal $\sR$ of finite-rank operators. Hence,
\begin{equation*}
 \|\dot{T}\|_{\dot{g}}:= \inf\big\{\|T+S\|_g;\ S\in \sR\big \}= \dist_{\sL_g}(T;\sR).
\end{equation*}
Therefore, using Lemma~\ref{lem:Lorentz.distance-formula}, we obtain
\begin{equation}\label{eq:quotient norm}
 \|\dot{T}\|_{\dot{g}}= \limsup_{j\rightarrow \infty} g(j)^{-1} \mu_j(T).
\end{equation}

\begin{lemma}\label{lem:r-triangle}
The quasi-norm $\|\cdot\|_{\dot{g}}$ is $r$-convex with $r:=(|\rho|+1)^{-1}$, i.e.,
\begin{equation}\label{eq:r-triangle}
 \|\dot{S}+\dot{T}\|_{\dot{g}}^r\leq  \|\dot{S}\|_{\dot{g}}^r+  \|\dot{T}\|_{\dot{g}}^r, \qquad S,T\in \sL_g.
\end{equation}
\end{lemma}
\begin{proof}
Let $S,T\in \sL_g$ and $\theta\in (0,1)$. Given $j\geq 0$, let $k_j$ be the unique integer such that $k_j\leq j\theta<k_j+1$. We thus have
\begin{equation*}
 j\theta -1 <k_j \leq j\theta \qquad \text{and} \qquad j(1-\theta)\leq j-k_j < j(1-\theta)+1.
\end{equation*}
As $k_j\sim \theta j$ and $g(t)$ is $\RV_\rho$, it follows from Lemma~\ref{lem:UCT-consequence} that
\begin{equation}\label{eq:key equality1}
 g(k_j)\sim g(\theta j)\sim \theta^\rho g(j).
\end{equation}
Likewise, as $j-k_j\sim (1-\theta)j$, we have
\begin{equation}\label{eq:key equality2}
 g(j-k_j)\sim g\left((1-\theta) j\right)\sim (1-\theta)^\rho g(j).
\end{equation}

Bearing this in mind, by Ky Fan's inequality~\eqref{eq:Lorentz.Fan1} we have
\begin{equation*}
 \mu_j (S+T) \leq  \mu_{k_j} (S) + \mu_{j-k_j} (T).
\end{equation*}
Thus,
\begin{equation*}
  g(j)^{-1}\mu_j (S+T) \leq  \frac{g(k_j)}{g(j)}\cdot g(k_j)^{-1}\mu_{k_j} (S) + \frac{g(j-k_j)}{g(j)} \cdot g(j-k_j)^{-1}\mu_{j-k_j} (T).
\end{equation*}
Taking $\limsup$ with respect to $j$ on both sides and using~(\ref{eq:key equality1})--(\ref{eq:key equality2}) yields
\begin{equation}\label{eq:key inequality}
\|\dot{S}+\dot{T}\|_{\dot{g}}  \leq \theta^{\rho} \|\dot{S}\|_{\dot{g}} + (1-\theta)^{\rho}\|\dot{T}\|_{\dot{g}}.
\end{equation}

If we set $r:=(|\rho|+1)^{-1}$, then we have
\begin{equation*}
 \min_{\theta \in (0,1)}\left\{ \theta^{\rho} \|\dot{S}\|_{\dot{g}} + (1-\theta)^{\rho}\|\dot{T}\|_{\dot{g}}\right\}=
 \big(\|\dot{S}\|_{\dot{g}}^r+\|\dot{T}\|_{\dot{g}}^r\big)^{\frac{1}{r}}.
\end{equation*}
Combining this with~(\ref{eq:key inequality}) gives~(\ref{eq:r-triangle}). The proof is complete.
\end{proof}

\begin{remark}\label{rmk:quasi-triangle-quotient}
 By~(\ref{eq:r-triangle}) the $r$-convexity of $\|\cdot\|_{\dot{g}}$ implies the quasi-triangular inequality,
 \begin{equation*}
\|\dot{S}+\dot{T}\|_{\dot{g}}  \leq 2^{|\rho|} \big(\|\dot{S}\big\|_{\dot{g}} + \|\dot{T}\|_{\dot{g}} \big), \qquad \dot{S},\dot{T}\in \dot{\sL}_g.
\end{equation*}
 \end{remark}

\begin{remark}\label{rmk:quasi-Banach-quotient}
If $F$ is a closed subspace of a quasi-Banach space $E$, then the quotient $E/F$ is a quasi-Banach space with respect to the quotient quasi-norm. The completeness of the quotient quasi-norm can be seen by using Remark~\ref{rmk:quasi-triangle-quotient} and arguing as in the Banach space setting (see, e.g.,~\cite[Proposition~11.8]{Br:Springer11}).
As $(\sL_g)_0$ is a closed subspace of $\sL_g$, it follows that the quasi-norm $\|\cdot\|_{\dot{g}}$ is complete. Therefore, we see that  $\dot{\sL}_g$ is an $r$-convex quasi-Banach space.
\end{remark}

\subsection{Quantitative estimates for eigenvalues}
For $T=T^*\in \sL_g$, combining the inequality~(\ref{eq:domi1}) with the definition of $\sL_g$, we obtain
\begin{equation}\label{eq:domi2}
 \lambda_j^\pm(T) \leq \mu_j(T) \leq  g(j)\|T\|_g \qquad \forall j\geq 0.
\end{equation}
We then define
\begin{equation}\label{eq:lower and upper}
 \overline{\Lambda}^\pm(T):= \limsup_{j\rightarrow \infty} g(j)^{-1}\lambda_j^\pm(T), \qquad
  \underline{\Lambda}^\pm(T):= \liminf_{j\rightarrow \infty} g(j)^{-1}\lambda_j^\pm(T).
\end{equation}
Note that
\begin{equation}\label{eq:weyl selfadjoint}
 \big( \lim_{j\rightarrow \infty} g(j)^{-1}\lambda_j^\pm(T)=c\big) \ \Longleftrightarrow \  \big(\underline{\Lambda}^\pm(T)=  \overline{\Lambda}^\pm(T)=c\big).
\end{equation}

\begin{lemma}\label{lem:cont-Lambda}
 Set $r=(|\rho|+1)^{-1}$. For all selfadjoint operators $S,T\in \sL_g$, we have
 \begin{gather}
  \left| \overline{\Lambda}^\pm(T)^{r} - \overline{\Lambda}^\pm(S)^{r} \right| \leq   \|\dot{T}-\dot{S}\|_{\dot{g}}^{r},\label{eq:upper}\\
  \left| \underline{\Lambda}^\pm(T)^{r} - \underline{\Lambda}^\pm(S)^{r} \right| \leq    \|\dot{T}-\dot{S}\|_{\dot{g}}^{r}.\label{eq:lower}
\end{gather}
\end{lemma}
\begin{proof}
 It is sufficient to prove the inequalities for the functionals $\overline{\Lambda}^+$ and $\underline{\Lambda}^+$, since replacing $(S,T)$ by $(-S,-T)$ gives the
 inequalities for $\overline{\Lambda}^-$ and $\underline{\Lambda}^-$. Moreover, by using Weyl's inequality~(\ref{eq:Bir-Sol.Weyl-Ineq}) and arguing as in the proof of Lemma~\ref{lem:r-triangle} we obtain the inequality,
\begin{equation*}
 \overline{\Lambda}^+(T)^{r}\leq \overline{\Lambda}^+(S)^{r} + \overline{\Lambda}^+(T-S)^{r}.
\end{equation*}
 It follows from~(\ref{eq:domi1}) that
 \begin{equation*}
 \overline{\Lambda}^+(T-S)\leq \limsup_{j\rightarrow \infty} g(j)^{-1}\mu_j(T-S) = \|\dot{T}-\dot{S}\|_{\dot{g}}.
\end{equation*}
Thus,
\begin{equation*}
 \overline{\Lambda}^+(T)^{r}- \overline{\Lambda}^+(S)^{r} \leq \|\dot{T}-\dot{S}\|_{\dot{g}}^{r}.
\end{equation*}
Interchanging the roles of $S$ and $T$ gives the inequality~(\ref{eq:upper}).

It remains to prove the inequality~(\ref{eq:lower}). It is sufficient to show that
\begin{equation}\label{eq:lower-semi}
 \underline{\Lambda}^+(T)^{r} \geq  \underline{\Lambda}^+(S)^{r}- \|\dot{T}-\dot{S}\|_{\dot{g}}^{r}.
\end{equation}
To prove this inequality we may further assume $ \underline{\Lambda}^+(S)^{r}> \|\dot{T}-\dot{S}\|_{\dot{g}}^{r}$, since otherwise the inequality holds trivially.

Let $\theta>1$. As in the proof of Lemma~\ref{lem:r-triangle}, for $j\geq 0$ we let $k_j$ be the unique integer such that $k_j\leq j\theta < k_j+1$.
The Weyl inequality~(\ref{eq:Bir-Sol.Weyl-Ineq}) gives
\begin{equation*}
 \lambda_{k_j}^+(S)\leq \lambda_j^+(T) + \lambda_{k_j-j}^+(S-T).
\end{equation*}
Thus,
\begin{equation*}
  g(j)^{-1}\lambda_j^+(T) \geq  \frac{g(k_j)}{g(j)} \cdot g(k_j)^{-1}\lambda_{k_j}^+(S) - \frac{g(k_j-j)}{g(j)}\cdot g(k_j-j)^{-1}\lambda_{k_j-j}^+(S-T).
\end{equation*}
Taking $\liminf$ with respect to $j$ on both sides and using~(\ref{eq:key equality1})--(\ref{eq:key equality2}) gives
\begin{align*}
  \underline{\Lambda}^+(T)&\geq \theta^\rho \underline{\Lambda}^+(S) - (\theta-1)^\rho \overline{\Lambda}^+(S-T)\\
  & \geq  \theta^\rho \underline{\Lambda}^+(S) - (\theta-1)^\rho \|\dot{T}-\dot{S}\|_{\dot{g}}.
\end{align*}
The inequality~(\ref{eq:lower-semi}) then follows from the fact that
\begin{equation*}
 \min_{\theta >1} \left\{\theta^\rho \underline{\Lambda}^+(S) - (\theta-1)^\rho \|\dot{T}-\dot{S}\|_{\dot{g}}\right\} =
\left( \underline{\Lambda}^+(S)^{r}- \|\dot{T}-\dot{S}\|_{\dot{g}}^{r}\right)^{\frac{1}{r}}.
\end{equation*}
The proof is complete.
\end{proof}

\begin{remark}
For weak Schatten classes the quantitative estimates~(\ref{eq:upper})--(\ref{eq:lower}) were obtained by Birman-Solomyak (see~\cite[Corollary~9.6.5]{BS:Book}). Therefore, Lemma~\ref{lem:cont-Lambda} extends those estimates to all weak Lorentz ideals.
\end{remark}

\subsection{Weyl operators}
In what follows we shall say that an operator $T=T^*\in \sL_g$ is a \emph{Weyl operator} if
\begin{equation*}
 \Lambda^+(T):= \lim_{j\rightarrow \infty} g(j)^{-1}\lambda_j^+(T) \quad \text{and} \quad \Lambda^-(T):= \lim_{j\rightarrow \infty} g(j)^{-1}\lambda_j^-(T)\quad \text{both exist}.
\end{equation*}
For instance, if $T=T^*\in (\sL_g)_0$, then it follows from Proposition~\ref{prop:separable part} that $T$ is a \emph{Weyl operator} with $\Lambda^\pm(T)=0$.

More generally, if $T\in \sL_g$ is not selfadjoint, we shall say that $T$ is a Weyl operator if its real part $\Re T:=\frac{1}{2}(T+T^*)$ and its imaginary part
$\Im T :=\frac{1}{2i}(T-T^*)$ are Weyl operators, in which case we set
\begin{equation*}
 \Lambda^\pm(T):= \Lambda^\pm\big(\Re(T)\big) + i \Lambda^\pm\big(\Im(T)\big).
\end{equation*}

\begin{definition}\label{def:weyl}
 $\sW(\sL_g)$ is the set of all Weyl operators in $\sL_g$.
\end{definition}

 We are now in a position to prove the following perturbation result for eigenvalue asymptotics for operators in $\sL_g$. This extends to weak Lorentz ideals the perturbation result of Birman-Solomyak~\cite{BS:JFAA70, BS:Book} for operators in weak Schatten classes.

\begin{theorem}\label{thm:cont-Lambda}
 Let $T\in \sL_g$ and $(T_\ell)_{\ell\geq 0}\subseteq \sW(\sL_g)$ be such that
 \begin{equation*}
  \dot{T}_\ell \longrightarrow \dot{T} \ \text{in}\ \dot{\sL}_g.
\end{equation*}
Then $\lim_{\ell \rightarrow \infty} \Lambda^\pm(T_\ell)$ both exist, $T\in  \sW(\sL_g)$, and we have
\begin{equation*}
 \Lambda^\pm(T) =\lim_{\ell \rightarrow \infty} \Lambda^\pm(T_\ell).
\end{equation*}
In particular, if $T$ and the operators $T_\ell$ are selfadjoint, then
\begin{equation*}
 \lim_{j \rightarrow \infty} g(j)^{-1} \lambda_j^\pm(T) =  \lim_{\ell \rightarrow \infty}\big(\lim_{j \rightarrow \infty} g(j)^{-1} \lambda_j^\pm(T_\ell) \big).
\end{equation*}
\end{theorem}
\begin{proof}
Let us first prove the result in the case where $T$ and the operators $T_\ell$ are selfadjoint. By~(\ref{eq:upper}) we have
\begin{equation*}
  \left| \overline{\Lambda}^\pm(T)^{r} - \Lambda^\pm(T_\ell)^{r} \right| \leq   \|\dot{T}-\dot{T}_\ell\|_{\dot{g}}^{r}.
\end{equation*}
As $\|\dot{T}-\dot{T}_\ell\|_{\dot{g}}\rightarrow 0$ as $\ell\rightarrow \infty$, we deduce that
\begin{equation}\label{eq:upper equal}
\lim_{\ell \rightarrow \infty} \Lambda^\pm(T_\ell) =  \overline{\Lambda}^\pm(T).
\end{equation}
Similarly, by using~(\ref{eq:lower}) we get
\begin{equation*}
\lim_{\ell \rightarrow \infty} \Lambda^\pm(T_\ell) =  \underline{\Lambda}^\pm(T).
\end{equation*}
Combining this with~(\ref{eq:upper equal}) proves the result in the selfadjoint case.

In general, by Lemma~\ref{lem:r-triangle} we have
\begin{equation*}
 \max\left\{ \|\Re \dot{T}\|_{\dot{g}}^r,\|\Im \dot{T}\|_{\dot{g}}^r\right\} \leq \frac{1}{2}\left(\|\dot{T}\|_{\dot{g}}^r + \|\dot{T}^*\|_{\dot{g}}^r \right) = \|\dot{T}\|_{\dot{g}}^r.
\end{equation*}
Therefore, if  $\dot{T}_\ell \rightarrow \dot{T}$ in  $\dot{\sL}_g$, then $\Re\dot{T}_\ell \rightarrow \Re \dot{T}$ and $\Im\dot{T}_\ell \rightarrow \Im \dot{T}$. The first part of the proof then ensures that $\lim_{\ell \rightarrow \infty} \Lambda^\pm(\Re T_\ell)$ and $\lim_{\ell \rightarrow \infty} \Lambda^\pm(\Im T_\ell)$ both exist, $\Re T$ and $\Im T$ are Weyl operators, and we have
\begin{equation*}
  \Lambda^\pm(\Re T) =\lim_{\ell \rightarrow \infty} \Lambda^\pm(\Re T_\ell) \qquad \text{and} \qquad  \Lambda^\pm(\Im T) =\lim_{\ell \rightarrow \infty} \Lambda^\pm(\Im T_\ell).
\end{equation*}
 This gives the result in the general case. The proof is complete.
\end{proof}

As immediate consequences of Theorem~\ref{thm:cont-Lambda} we get the following results.

\begin{corollary}\label{cor:Bir-Sol.continuity-Lambda}
 $\sW(\sL_g)$ is a closed subset of $\sL_g$ on which $\Lambda^\pm$ are continuous functionals.
\end{corollary}

\begin{corollary}\label{cor:ptb-Lambda}
  If $T\in \sW(\sL_g)$ and $S\in (\sL_g)_0$, then $T+S\in  \sW(\sL_g)$, and
 \begin{equation*}
\Lambda^\pm(T+S)=\Lambda^\pm(T).
\end{equation*}
In particular, if $T$ and $S$ are selfadjoint, then
\begin{equation*}
 \lim_{j\rightarrow \infty} g(j)^{-1}\lambda_j^\pm(T+S)=  \lim_{j\rightarrow \infty} g(j)^{-1}\lambda_j^\pm(T).
\end{equation*}
\end{corollary}

\begin{remark}
 For weak Schatten classes the previous result goes back to Weyl~\cite[Satz~V]{We:NGWG11} (see also~\cite[\S9.6]{BS:Book}). Therefore, Corollary~\ref{cor:ptb-Lambda} extends Weyl's result to weak Lorentz ideals.
\end{remark}

\subsection{Singular values}
The perturbation theory for positive and negative eigenvalues extends to singular values as follows.

\begin{definition}
 $|\sW|(\sL_g)$ consists of operators $T\in \sL_g$ such that
 \begin{equation*}
 \Delta(T):=\lim_{j\rightarrow \infty}g(j)^{-1}\mu_j(T)\ \text{exists}.
\end{equation*}
\end{definition}

As $\mu_j(T)=\lambda_j(|T|)=\lambda_j^+(|T|)$, we see that
\begin{equation*}
 \big( \lim_{j\rightarrow \infty}g(j)^{-1}\mu_j(T)\ \text{exists}\big) \Longleftrightarrow |T| \in \sW(\sL_g).
\end{equation*}
That is, $ T\in |\sW|(\sL_g)$ if and only if $|T|$ is a Weyl operator.

In what follows, for $T\in \sL_g$ we set
\begin{equation*}
 \overline{\Delta}(T)= \limsup_{j\rightarrow \infty} g(j)^{-1}\mu_j(T) \qquad \text{and} \qquad \underline{\Delta}(T)= \liminf_{j\rightarrow \infty} g(j)^{-1}\mu_j(T).
\end{equation*}
In particular, by~(\ref{eq:quotient norm}) we have $\overline{\Delta}(T)=\|\dot{T}\|_{\dot{g}}$. Moreover, as with~(\ref{eq:weyl selfadjoint}) we have
\begin{equation*}
 \big( T\in |\sW|(\sL_g) \ \text{and}\ \Delta(T)=c\big) \ \Longleftrightarrow \  \big(\overline{\Delta}(T)=  \underline{\Delta}(T)=c\big).
\end{equation*}

We have the following version of Lemma~\ref{lem:cont-Lambda} for singular values.

\begin{lemma}\label{lem:cont-Delta}
 Set $r=(|\rho|+1)^{-1}$. For all $S,T\in \sL_g$, we have
 \begin{gather}
  \left| \overline{\Delta}(T)^{r} - \overline{\Delta}(S)^{r} \right| \leq   \|\dot{T}-\dot{S}\|_{\dot{g}}^{r},\label{eq:upper quantitative}\\
  \left| \underline{\Delta}(T)^{r} - \underline{\Delta}(S)^{r} \right| \leq    \|\dot{T}-\dot{S}\|_{\dot{g}}^{r}. \label{eq:lower quantitative}
\end{gather}
\end{lemma}
\begin{proof}
 The first inequality is a mere restatement of~(\ref{eq:upper}). The second inequality is proved by using Fan's inequality~(\ref{eq:Lorentz.Fan1}) and arguing as in the second part of the proof of Lemma~\ref{lem:cont-Lambda}.
\end{proof}

\begin{remark}
For weak Schatten classes the quantitative estimates~(\ref{eq:upper quantitative})--(\ref{eq:lower quantitative}) were established by Birman-Solomyak (see~\cite[Corollary~9.6.5]{BS:Book}). Therefore, as with Lemma~\ref{lem:cont-Lambda}, we see that Lemma~\ref{lem:cont-Delta} extends Birman-Solomyak's estimates to all weak Lorentz ideals.
\end{remark}

By using Lemma~\ref{lem:cont-Delta} and arguing as in the proof of Theorem~\ref{thm:cont-Lambda}, we obtain the following perturbation result for asymptotics of singular values for operators in $\sL_g$. This extends to weak Lorentz ideals the perturbation result of Birman-Solomyak for operators in weak Schatten classes.

\begin{theorem}\label{thm:Bir-Sol.Delta}
 Let $T\in \sL_g$ and $(T_\ell)\subseteq |\sW|(\sL_g)$ be such that
 \begin{equation*}
  \dot{T}_\ell \longrightarrow \dot{T} \ \text{in}\ \dot{\sL}_g.
\end{equation*}
Then $\lim_{\ell \rightarrow \infty} \Delta(T_\ell)$ exists, $T\in  |\sW|(\sL_g)$, and we have
\begin{equation*}
\lim_{j\rightarrow \infty} g(j)^{-1}\mu_j(T) = \lim_{\ell \rightarrow \infty}  \Delta(T_\ell)=   \lim_{\ell \rightarrow \infty}
\big(\lim_{j\rightarrow \infty} g(j)^{-1}\mu_j(T_\ell)\big).
\end{equation*}
\end{theorem}

In particular, we have the following consequences of Theorem~\ref{thm:Bir-Sol.Delta}.

\begin{corollary}\label{cor:Bir-Sol.continuity-Delta}
 $|\sW|(\sL_g)$ is a closed subset of $\sL_g$ on which the functional $\Delta$ is continuous.
\end{corollary}

\begin{remark}
 The above result can also be obtained by combining~(\ref{eq:Bir-Sol.mu-nu}) below with~\cite[Lemma~3.2]{Ro:MUS74}. 
\end{remark}

\begin{corollary}\label{cor:Bir-Sol.Ky-Fan-Lorentz}
 If $T\in |\sW|(\sL_g)$ and $S\in (\sL_g)_0$, then $T+S\in  |\sW|(\sL_g)$, and
 \begin{equation*}
\lim_{j\rightarrow \infty} g(j)^{-1}\mu_j(T+S) =\lim_{j\rightarrow \infty} g(j)^{-1}\mu_j(T).
\end{equation*}
\end{corollary}

\begin{remark}\label{rmk:Bir-Sol.Gohberg-Krein}
 For weak Schatten classes the previous result goes back to Ky~Fan~\cite[Theorem~3]{Fa:PNAS51} (see also~\cite[Theorem~II.2.3]{GK:AMS69}). It was also conjectured by Gohberg-Krein~\cite[pp.~32--33]{GK:AMS69} that Fan's result holds for weak Lorentz ideals. Therefore, Corollary~\ref{cor:Bir-Sol.Ky-Fan-Lorentz} extends Fan's theorem to all weak Lorentz ideals and gives a positive answer to Gohberg-Krein's conjecture.
\end{remark}

\subsection{Counting functions} The perturbation results of Birman-Solomyak can be equivalently reformulated in terms of counting functions (see, e.g., \cite[\S11.6]{BS:Book}). For the sake of completeness, we explain how to reformulate the main results of this section in terms of counting functions.

For $T=T^*\in \sL_g$, we set
\begin{equation*}
 N^\pm(T;\lambda) = \#\left\{j; \ \lambda_j^\pm(T)>\lambda\right\}, \qquad \lambda>0.
\end{equation*}
For $T\in \sL_g$, we also define
\begin{equation*}
 \nu(T;\lambda) = \#\left\{j; \ \mu_j(T)>\lambda\right\}, \qquad \lambda>0.
\end{equation*}

Set $h(t)=g(t)^{-1}$, $t\geq 0$; this is an $\RV_{|\rho|}$-function. We also let $h^\sharp:[0,\infty)\rightarrow (0,\infty)$ be an asymptotic $\RV_{|\rho|^{-1}}$-inverse, i.e., an $\RV_{|\rho|^{-1}}$ function such that
 \begin{equation*}
 h(h^\sharp(t)) \sim t \qquad \text{and} \qquad h^\sharp(h(t)) \sim t  \qquad \text{as}\ t\rightarrow \infty.
 \end{equation*}
Such a function always exists (see~Remark~\ref{rmk:B2}).

For instance, if $g(t)=(t+1)^{-1/p}$, with $p=-1/\rho>0$, then we may take $h^\sharp(t)=(1+t)^{p}$. If $g(t)=(t+1)^{-1/p}(\log (t+2))^{q}$, $q\in \R$, then $h(t)= (t+1)^{1/p}(\log (t+2))^{-q}$, and so we may take $h^\sharp(t)=p^{pq}\, t^{p} \left(\log (t+2)\right)^{pq}$, $t\geq 0$ (\emph{cf}.~Example~\ref{ex:asymp inverse}).

By Lemma~\ref{lem:B5}, if $T=T^*\in \sL_g$, then
\begin{equation*}
 \lim_{j\rightarrow \infty} g(j)^{-1}\lambda_j^\pm(T) =c^\pm
 \ \Longleftrightarrow \
  \lim_{\lambda\rightarrow 0^+} h^\sharp(\lambda^{-1})^{-1}N^\pm(T;\lambda)=\left(c^\pm\right)^{\frac{1}{|\rho|}}.
\end{equation*}
In particular, we see that $T\in \sW(\sL_g)$ if and only if $ \lim_{\lambda\rightarrow 0^+} h^\sharp(\lambda^{-1})^{-1}N^\pm(T;\lambda)$ exist. Therefore, Theorem~\ref{thm:cont-Lambda} implies the following result.

\begin{proposition}\label{prop:Bir-Sol.Npm}
 Let $T=T^*\in \sL_g$ and $T_\ell=T_\ell^*\in \sW(\sL_g)$, $\ell\geq 0$, be such that
  \begin{equation*}
  \dot{T}_\ell \longrightarrow \dot{T} \ \text{in}\ \dot{\sL}_g.
\end{equation*}
Then $\lim_{\ell \rightarrow \infty} (\lim_{\lambda\rightarrow 0^+} h^\sharp(\lambda^{-1})^{-1}N^\pm(T_\ell;\lambda))$ exist, and we have
\begin{equation*}
 \lim_{\lambda\rightarrow 0^+} h^\sharp(\lambda^{-1})^{-1}N^\pm(T;\lambda) = \lim_{\ell \rightarrow \infty} \big(\lim_{\lambda\rightarrow 0^+} h^\sharp(\lambda^{-1})^{-1}N^\pm(T_\ell;\lambda)\big).
\end{equation*}
\end{proposition}

Similarly, for any $T\in \sL_g$, Lemma~\ref{lem:B5} also implies that
\begin{equation}\label{eq:Bir-Sol.mu-nu}
 \lim_{j\rightarrow \infty} g(j)^{-1}\mu_j(T) =c
 \ \Longleftrightarrow \
  \lim_{\lambda\rightarrow 0^+} h^\sharp(\lambda^{-1})^{-1}\nu(T;\lambda)=c^{\frac{1}{|\rho|}}.
\end{equation}
In particular, we see that $T\in |\sW|(\sL_g)$ if and only if $ \lim_{\lambda\rightarrow 0^+} h^\sharp(\lambda^{-1})^{-1}\nu(T;\lambda)$ exists. We thus have the following equivalent reformulation of Theorem~\ref{thm:Bir-Sol.Delta}.

\begin{proposition}\label{prop:Bir-Sol.nu}
 Let $T\in \sL_g$ and $T_\ell\in |\sW|(\sL_g)$, $\ell\geq 0$, be such that
  \begin{equation*}
  \dot{T}_\ell \longrightarrow \dot{T} \ \text{in}\ \dot{\sL}_g.
\end{equation*}
Then $\lim_{\ell \rightarrow \infty} (\lim_{\lambda\rightarrow 0^+} h^\sharp(\lambda^{-1})^{-1}\nu(T_\ell;\lambda))$ exists, and we have
\begin{equation*}
 \lim_{\lambda\rightarrow 0^+} h^\sharp(\lambda^{-1})^{-1}\nu(T;\lambda) = \lim_{\ell \rightarrow \infty} \big(\lim_{\lambda\rightarrow 0^+} h^\sharp(\lambda^{-1})^{-1}\nu(T_\ell;\lambda)\big).
\end{equation*}
\end{proposition}

\begin{remark}
The above result implies that $T\rightarrow  \lim_{\lambda\rightarrow 0^+} h^\sharp(\lambda^{-1})^{-1}\nu(T;\lambda)$ is a continuous functional on $ |\sW|(\sL_g)$. 
This continuity property is a special case of a result of Rozenblum~\cite[Lemma~3.2]{Ro:MUS74}. 
\end{remark}

\section{Strong Measurability}\label{sec:strong-hyper}
From the point of view of noncommutative geometry, the NC integral should be a positive trace. Dixmier traces are (normalized)  positive traces (see below for the definition of normalized traces). There is a whole wealth of normalized positive traces that are not Dixmier traces (see, e.g.,~\cite{LSZ:Book, Us:SM21}). For this reason,  as with the weak trace class $\sL_{1,\infty}$, it stands to reason to look at a stronger notion of measurability in terms of the whole class of positive normalized traces.

In this section, we investigate this notion and relate it to the spectral asymptotics considered in the previous section. This will lead to a further notion of measurability.

Throughout this section, we let $g:[0,\infty) \rightarrow (0,\infty)$ be a continuous $\RV_{-1}$-function such that $\int_0^\infty g(t)dt=\infty$.

\subsection{Strong measurability}
In what follows, we let $T_{g}$ be any compact operator such that
\begin{equation}\label{eq:generator}
 T_{g}\xi_j = g(j)\xi_j, \qquad j\geq 0,
\end{equation}
where $(\xi_j)_{j \geq 0}$ is an arbitrary orthonormal basis of $\sH$. This ensures that $T_{g}$ is a positive operator in $\sL_g$ such that
\begin{equation*}
 \mu_j(T_{g})=\lambda_j(T_{g}) = g(j) \qquad \text{for}\ j \gg 1.
\end{equation*}

In what follows, a trace  $\varphi:\sL_g\rightarrow \C$ is called \emph{normalized} (or \emph{$g$-normalized}) if $\varphi(T_{g})=1$. For instance, every Dixmier trace is normalized.

Note that in view of~(\ref{eq:uni invariance}) this normalization does not depend on the choice of the orthonormal basis $(\xi_j)_{j\geq 0}$. However, it may depend on the choice of the function $g$ to define the weak Lorentz ideal $\sL_g$ (see~\cite{Po-Ti:Part2}). It can be shown, however, that for positive traces (and even continuous traces) the normalization $\varphi(T_g)=1$ does not depend on $g$ (see Lemma~\ref{lem:independent of generator}).

\begin{definition}
 An operator $T\in \sL_g$ is said to be \emph{strongly measurable} (or \emph{positively measurable}) if all normalized positive traces take the same value on $T$.
\end{definition}

We shall denote by $\sM_s(\sL_g)$ the space of strongly measurable operators in $\sL_g$.

\begin{remark}\label{rmk:strong implies dixmier}
If $T\in \sL_g$ is strongly measurable, then $T$ is measurable. Moreover, for every normalized positive trace $\varphi:\sL_g\rightarrow \C$, we have
\begin{equation}\label{eq:normalized to nci}
 \varphi(T) = \gbint T.
\end{equation}
\end{remark}

The following statement (which is essentially proved in~\cite{LSZ:Book}) gives the relationship between continuous traces and normalized positive traces.

\begin{proposition}\label{prop:cont--positive}
 The following hold.
 \begin{enumerate}
 \item Every positive trace on $\sL_g$ is continuous and is a scalar multiple of a normalized positive trace.

 \item Every continuous trace on $\sL_g$ is a linear combination of four normalized positive traces.
\end{enumerate}
\end{proposition}
\begin{proof}
 By~\cite[Lemma~2.4.12]{LSZ:Book} every positive linear functional on a quasi-Banach ideal is continuous. By~\cite[Theorem~4.1.10]{LSZ:Book} every continuous trace on a quasi-Banach ideal is a linear combination of 4 positive traces. Moreover, in the same way as in the proof of~\cite[Lemma~4.3]{Po:JNCG23}, it can be shown that if $\varphi:\sL_g \rightarrow  \C$ is a non-zero positive trace, then $\varphi(T_{g})> 0$. Thus, $\psi:= \varphi(T_{g})^{-1}\varphi$ is a normalized positive trace such that $\varphi=\varphi(T_{g})\psi$. This shows that every positive trace is a scalar multiple of a normalized positive trace. Combining this with the first part of the proof shows that every continuous trace is a linear combination of four normalized positive traces. The proof is complete.
\end{proof}

We have the following description of the space of strongly measurable operators.

\begin{proposition}\label{prop:decomp}
 We have
 \begin{equation*}
 \sM_s(\sL_g)=\C T_{g} \oplus \overline{\Com\big(\sL_g\big)}.
\end{equation*}
In particular, $\sM_s(\sL_g)$ is a closed subspace of $\sL_g$. Moreover, if $T\in \sM_s(\sL_g)$, then
\begin{equation}\label{eq:nci as normalization factor}
  \varphi(T) = \bigg(\gbint T \bigg)\varphi(T_{g}) \qquad \text{for every continuous trace  $\varphi:\sL_g\rightarrow \C$}.
\end{equation}
\end{proposition}
\begin{proof}
 If $T\in \sM_s(\sL_g)$, then by Remark~\ref{rmk:strong implies dixmier} the equality~(\ref{eq:nci as normalization factor}) holds for all normalized positive traces on $\sL_g$. By Proposition~\ref{prop:cont--positive} the space of continuous traces on $\sL_g$ is spanned by normalized positive traces. Therefore, by linearity the equality~(\ref{eq:nci as normalization factor}) holds for all continuous traces on $\sL_g$. Thus, if we set $c=\gbint T$, then $T-cT_{g}$ is annihilated by all continuous traces, and hence is contained in $\overline{\Com(\sL_g)}$. Thus, $T\in \C T_{g} \oplus \overline{\Com(\sL_g)}$. This shows that  $  \sM_s(\sL_g) \subseteq \C T_{g} \oplus \overline{\Com\big(\sL_g\big)}$.

 It remains to show that $ \C T_{g} \oplus \overline{\Com\big(\sL_g\big)} \subseteq\sM_s(\sL_g)$. Suppose that $T\in \C T_{g} \oplus \overline{\Com(\sL_g)}$, i.e., $T=c T_{g}$ for some $c\in \C$. If $\varphi$ is a normalized positive trace on $\sL_g$, then it is continuous by Proposition~\ref{prop:cont--positive}, and so it annihilates $\overline{\Com(\sL_g)}$. As $\varphi(T_{g})=1$, we get
\begin{equation*}
 \varphi(T)=\varphi(cT_{g})=c\varphi(T_{g})=c.
\end{equation*}
 This shows that $T$ is strongly measurable, and hence $ \C T_{g} \oplus \overline{\Com\big(\sL_g\big)} \subseteq\sM_s(\sL_g)$.  The proof is complete.
 \end{proof}

Recall that a trace  $\varphi:\sL_g\rightarrow \C$ is called \emph{singular} if it annihilates finite-rank operators.

\begin{proposition}\label{prop:strong-meas.positive-trace-sLg0}
 Every positive trace on $\sL_g$ annihilates $(\sL_g)_0$, and hence is singular. In particular, every operator in $(\sL_g)_0$ is strongly measurable.
\end{proposition}

\begin{proof}
By Proposition~\ref{prop:cont--positive} every positive trace is a constant multiple of a normalized positive trace. Therefore, it is enough to prove the result for normalized positive traces. Thus, let $\varphi$ be a normalized positive trace. If $\xi$ and $\eta$ are two unit vectors in $\sH$ and $U$ is a unitary operator such that $U\xi=\eta$,
then $\ketbra{\eta}{\eta}=\ketbra{U\xi}{U\xi}=U\ketbra{\xi}{\xi}U^*$. Thus,
\begin{equation*}
 \varphi\big(\ketbra{\eta}{\eta}\big)= \varphi\big(U\ketbra{\xi}{\xi}U^*\big)=\varphi\big(\ketbra{\xi}{\xi}\big).
\end{equation*}
This shows that $\varphi$ takes the same value on all rank 1 orthogonal projections. Denote this value by $c$.

We claim that $c=0$. To see this, let $T_{g}=\sum_{j\geq 0} g(j) \ketbra{\xi_j}{\xi_j}$ be a Schmidt series representation for $T_{g}$. As $\varphi$ is positive and normalized, for all $N\geq 1$, we have
\begin{equation*}
 1=\varphi(T_{g}) \geq \varphi\big( \sum_{j<N}g(j)\ketbra{\xi_j}{\xi_j}\big)=\sum_{j<N}g(j)\varphi\big(\ketbra{\xi_j}{\xi_j}\big)=c\sum_{j<N}g(j).
\end{equation*}
This implies that $\sum_{j\geq 0}g(j)\leq c^{-1}<\infty$, which contradicts~(\ref{eq:Lorentz.divergence}). Therefore, $\varphi$ annihilates all rank 1 orthogonal projections. By linearity it annihilates all finite rank operators. We know from Proposition~\ref{prop:cont--positive} that $\varphi$ is continuous. Therefore, it annihilates the closure in $\sL_g$ of finite rank operators, i.e., the ideal $(\sL_g)_0$. The proof is complete.
\end{proof}

\begin{remark}
 More generally, for any ideal $\sI$ which is not contained in $\sL_1$, it can be shown that every positive trace is singular (see~\cite[Lemma~15]{SZ:AIM14}). In fact, the proof given above essentially follows the outline of the proof of~\cite[Lemma~15]{SZ:AIM14}.
\end{remark}

\begin{remark}\label{rmk:spt vs cl-com}
 It follows from Proposition~\ref{prop:strong-meas.positive-trace-sLg0} and Proposition~\ref{prop:cont--positive} that $(\sL_g)_0$ is annihilated by all continuous traces. Thus,
 \begin{equation*}
 (\sL_g)_0 \subseteq \overline{\Com(\sL_g)}.
\end{equation*}
\end{remark}

\subsection{Strong measurability and Weyl laws}
As the following result shows, every Weyl operator in $\sL_g$ is strongly measurable. This extends to weak Lorentz ideals a result for the weak trace class $\sL_{1,\infty}$  (see, e.g., \cite[Proposition~5.12]{Po:JNCG23}).

\begin{theorem}\label{thm:weyl implies strong}
If $T\in \sW(\sL_g)$, then $T$ is strongly measurable,  and we have
\begin{equation*}
 \gbint T = \Lambda^+(T) - \Lambda^{-}(T).
\end{equation*}
In particular, if $T$ is selfadjoint, then
\begin{equation*}
  \gbint T = \lim_{j\rightarrow \infty} g(j)^{-1} \lambda_j^+(T)- \lim_{j\rightarrow \infty} g(j)^{-1} \lambda_j^-(T).
\end{equation*}
\end{theorem}
\begin{proof}
By linearity it is sufficient to prove the result for $T\geq 0$. As $T$ is a Weyl operator, we have $\lambda_j(T)=\Lambda(T)g(j)+\op{o}(g(j))$. We deduce from~(\ref{eq:Lorentz.divergence}) that
$\sum_{j<N}g(j)-G(N)=\op{O}(1)=\op{o}(G(N))$. Thus,
\begin{equation*}
 \sum_{j<N} \lambda_j(T)= \Lambda(T)\sum_{j<N}g(j) +\op{o}\big(\sum_{j<N}g(j)\big)=\Lambda(T)G(N)+\op{o}(G(N)).
\end{equation*}
Theorem~\ref{thm:measurable} then ensures that $T$ is measurable and $\gbint T=\Lambda(T)$.

It remains to show that $T$ is strongly measurable. In fact, by proceeding as in the proof of~\cite[Proposition~5.12]{Po:JNCG23}, it can be shown that $T-\Lambda(T)T_{g}\in (\sL_g)_0$. By Proposition~\ref{prop:strong-meas.positive-trace-sLg0} operators in $(\sL_g)_0$ are strongly measurable. As $T_{g}$ is strongly measurable, it follows that $T$ is strongly measurable. This completes the proof.
\end{proof}

\begin{corollary}
 If $T\in |\sW|(\sL_g)$, then $|T|$ is strongly measurable,  and we have
 \begin{equation*}
  \gbint |T| = \lim_{j\rightarrow \infty} g(j)^{-1} \mu_j(T).
\end{equation*}
\end{corollary}

Theorem~\ref{thm:weyl implies strong} asserts that every Weyl operator in $\sL_g$ is strongly measurable. The converse does not hold (see~\cite{Po-Ti:Part2}). This leads us to introduce the following notion of measurability, which then is a stronger property than strong measurability.

\begin{definition}
 An operator $T\in \sL_g$ is called \emph{spectrally measurable} if it is a Weyl operator in the sense of Definition~\ref{def:weyl}, i.e., it belongs to $\sW(\sL_g)$.
\end{definition}

\subsection{Dependence on the function $g(t)$}
As a further consequence of Theorem~\ref{thm:weyl implies strong}, we can show that the space $\sM_s(\sL_g)$ does not depend on the choice of the function $g(t)$ to define $\sL_g$.

Let $g_1:[0,\infty)\rightarrow (0,\infty)$ be some (continuous) $\RV_{-1}$-function such that $g_1(t)\sim g(t)$ as $t\rightarrow \infty$. In this case the quasi-Banach ideals $\sL_{g_1}$ and $\sL_g$ agree and have equivalent quasi-norms. We denote by $\sM_s^g(\sL_g)$ (resp., $\sM_s^{g_1}(\sL_g)$) the space of strongly measurable operators associated with the normalization $\varphi(T_g)=1$ (resp., $\varphi(T_{g_1})=1$).

We have the following consequence of Theorem~\ref{thm:weyl implies strong}.

\begin{lemma}\label{lem:independent of generator}
 For every continuous trace $\varphi$ on $\sL_g$ we have $\varphi(T_{g_1})=\varphi(T_g)$.
\end{lemma}
\begin{proof}
 For $j$ large enough we have $\lambda_j(T_{g_1})=g_1(j)\sim g(j)$. Thus, $T_{g_1}$ is a (positive) Weyl operator. Theorem~\ref{thm:weyl implies strong} then ensures that $T_{g_1}$ belongs to $\sM_s^g(\sL_g)$, and we have
 \begin{equation*}
 \gbint T_{g_1}= \lim_{j\rightarrow \infty} g(j)^{-1} \lambda_j(T_{g_1})=1.
\end{equation*}
Combining this with Proposition~\ref{prop:decomp} we deduce that, for every continuous trace $\varphi$ on $\sL_g$, we have
\begin{equation*}
 \varphi\left(T_{g_1}\right)= \bigg( \gbint T_{g_1}\bigg)\varphi(T_g)= \varphi(T_g).
\end{equation*}
The result is proved.
\end{proof}

It follows from Lemma~\ref{lem:independent of generator} that, for positive traces on $\sL_g$, the normalization conditions $\varphi(T_g)=1$ and $\varphi(T_{g_1})=1$ are equivalent. As a result, the set of $g_1$-normalized positive traces agrees with that of $g$-normalized positive traces. We thus arrive at the following result.

\begin{proposition}\label{prop:strong-independent}
We have $\sM_s^{g_1}(\sL_{g})=\sM_s^g(\sL_{g})$.
\end{proposition}

This shows that $\sM_s(\sL_g)$ depends only on the quasi-Banach ideal $\sL_g$, not on the choice of function in the equivalence class of $g(t)$.

\begin{remark}
A full spectral characterization of strongly measurable operators in $\sL_g$, analogous to the spectral characterization of measurability in Theorem~\ref{thm:measurable}, will be established in the companion paper~\cite{Po-Ti:Part2} via Pietsch's correspondence for traces on~$\sL_g$.
\end{remark}

\section{Examples Arising from Nonclassical Weyl Laws}\label{sec:examples}
In this section, we describe various concrete examples of spectrally measurable and $g$-hypermeasurable operators, and compute their NC integrals. These examples come from a variety of settings. Most of them arise from nonclassical Weyl laws in the sense of Simon~\cite{Si:JFA83}. We also describe a general recipe to construct examples from operators satisfying nonclassical Weyl laws.

Further examples can be found in~\cite{GS:JFA14, GU:JMAA20, LU:JMAA25, TU:JMAA27}. We stress that in these references the focus is mostly on Dixmier measurability, while in this section we focus on spectral measurability and $g$-hypermeasurability.

\subsection{General construction of examples}
Let $A:\dom(A)\rightarrow \sH$ be a selfadjoint operator with non-negative spectrum such that
\begin{itemize}
 \item $0$ is isolated in the spectrum of $A$

 \item The non-zero part of the spectrum of $A$ consists of isolated eigenvalues with finite multiplicity.
\end{itemize}
We then can arrange the positive eigenvalues of $A$ as a non-decreasing sequence,
\begin{equation*}
 \lambda_0(A)\leq  \lambda_1(A) \leq \lambda_2(A)\leq \cdots
\end{equation*}
where each eigenvalue is repeated according to multiplicity. We then define the counting function of $A$ by
\begin{equation*}
 N(A;\lambda)=\#\left\{j; \ \lambda_j\left(A\right) <\lambda\right\},  \qquad \lambda>0.
\end{equation*}

For $p>0$ we define $A^{-p}$ to be $f(A)$ with $f(t)=\car_{(0,\infty)}(t)t^{-p}$.  Equivalently, if we let $(\xi_j)_{j\geq 0}$ be any orthonormal family in $\sH$ such that $A\xi_j=\lambda_j(A)\xi_j$ for all $j\geq 0$, then we have
\begin{equation*}
 A^{-p}=0 \quad \text{on}\ \ker A \qquad \text{and} \qquad A^{-p}\xi_j= \lambda_j(A)^{-p} \xi_j \quad \text{for all}\ j\geq 0.
\end{equation*}
In particular, $A^{-p}$ is a positive compact operator, and we have
\begin{equation*}
 \lambda_{j}(A^{-p})=\lambda_{j}(A)^{-p} \qquad \forall j\geq 0.
\end{equation*}
Thus, if we define the counting function of $A^{-p}$ by
\begin{equation*}
 N\left(A^{-p};\lambda\right)=\#\left\{j; \ \lambda_j\big(A^{-p}\big) >\lambda\right\},  \qquad \lambda>0,
\end{equation*}
then we have
\begin{equation}\label{eq:example.NA-NAp}
 N\left(A^{-p};\lambda\right)= N\big(A;\lambda^{-\frac{1}{p}}\big) \qquad \forall \lambda>0.
\end{equation}

 We focus on cases where $A$ satisfies a (nonclassical) Weyl law of the form,
\begin{equation}\label{eq:example.A-Weyl-law}
 N(A;\lambda)\sim c \lambda^p (\log \lambda)^q, \qquad c>0, \quad p>0, \quad q\geq -1.
\end{equation}
We then set $g(t)=(t+1)^{-1}(\log(t+2))^q$, $t\geq 0$. This is a continuous $\RV_{-1}$-function. Note that for $q=0$ the ideal $\sL_g$ is just the weak trace class $\sL_{1,\infty}$.

\begin{lemma}\label{lem:example.Ap-spectral-meas}
If $A$ satisfies the Weyl law~(\ref{eq:example.A-Weyl-law}), then the operator $A^{-p}$ is spectrally measurable in $\sL_{g}$, and we have
 \begin{equation}\label{eq:example.Ap-int-formula}
 \gbint A^{-p} = cp^{-q}.
\end{equation}
\end{lemma}
\begin{proof}
It follows from~(\ref{eq:example.NA-NAp})--(\ref{eq:example.A-Weyl-law}) that, as $\lambda \rightarrow 0^+$, we have
\begin{equation*}
  N\left(A^{-p};\lambda\right)= N\big(A,\lambda^{-\frac{1}{p}}\big) \sim c \lambda^{-1} \left(\log \big(\lambda^{-\frac1p}\big)\right)^q = c p^{-q}\lambda^{-1} |\log  \lambda|^q.
\end{equation*}
Combining this with Lemma~\ref{lem:B5} gives the eigenvalue asymptotics, 
 \begin{equation*}
 \lambda_j\big(A^{-p}\big) \sim c p^{-q} j^{-1} (\log j)^q \qquad \text{as}\ j \rightarrow \infty. 
\end{equation*}
The result then follows from Theorem~\ref{thm:weyl implies strong}. 
\end{proof}

\subsection{Zeros of Riemann's zeta function (Riemann Hypothesis)}
Assume the Riemann Hypothesis (RH) holds, and let $\sD$ be the ``Dirac operator'' whose spectrum consists of the imaginary parts of non-trivial zeros of the Riemann zeta function  $\zeta(s):=\sum_{n\geq 1} n^{-s}$ (see~\cite{Co:AFA24, CM:PNAS22}). Thus, $\sD$ is a selfadjoint unbounded operator on some suitable Hilbert space.

As the zeros of $\zeta(s)$ are symmetric with respect to the real axis, we see that the counting function $N(|\sD|;\lambda)$, $\lambda>0$, is two times the number of zeros of
$\zeta(s)$ with imaginary part in $[0,\lambda)$. Therefore, by the Riemann-von Mangoldt formula (see, e.g., \cite{Ti:Zeta}), as $\lambda\rightarrow \infty$ we have
\begin{align*}
 N\big(|\sD|;\lambda\big) & = 2\frac{\lambda}{2\pi}\left( \log\left(\frac{\lambda}{2\pi}\right)-1\right)+\op{O}(\log \lambda)\\
 & = \frac{1}{\pi} \lambda \log \lambda\left( 1+ \op{O}\left((\log \lambda)^{-1} \right)\right).
\end{align*}
Applying Lemma~\ref{lem:example.Ap-spectral-meas} then gives the following result.

\begin{proposition}\label{prop:RH}
 Assume (RH), and set $g(t)=(t+1)^{-1}\log (t+2)$, $t\geq 0$. The operator $|\sD|^{-1}$ is spectrally measurable in $\sL_g$, and we have
 \begin{equation*}
 \gbint |\sD|^{-1} = \frac{1}{\pi}.
\end{equation*}
\end{proposition}

\begin{remark}
 As $|\sD|$ is a positive operator, the fact that $|\sD|^{-1}$ is spectrally measurable means that 
 \begin{equation}\label{eq:asymptotic of RH dirac}
 \lambda_j\left(|\sD|^{-1}\right)\sim  \frac{1}{\pi} j^{-1} (\log j).
\end{equation}
 The symmetry between positive and negative eigenvalues of $\sD$ ensures that $\lambda_j^\pm(\sD^{-1})=\lambda_{2j}(|\sD|^{-1})$ for all $j\geq 0$. Therefore, the positive and negative eigenvalues satisfy asymptotics of the form~(\ref{eq:asymptotic of RH dirac}). It then follows that $\sD^{-1}$ is spectrally measurable in $\sL_g$. In fact, the symmetry further ensures that $\sH$ has an orthogonal splitting $\sH=\sH^+\oplus \sH^-$, and we can find an orthonormal eigenbasis $(\xi_j^\pm)$ of $\sH^{\pm}$ in such a way that $\sD^{-1}\xi_j^\pm =\pm \lambda_{2j}(|\sD|)\xi_j^\pm$. If $J$ is the involution on $\sH$ such that $J\xi_j^\pm=\xi_j^\mp$, then we have $J\sD^{-1}J=-\sD^{-1}$. Equivalently,
 \begin{equation*}
 \sD^{-1}= \frac{1}{2}\big(\sD^{-1}-J\sD^{-1}J\big)= \frac12\big[J,J\sD^{-1}\big].
\end{equation*}
This shows that $\sD^{-1}$ is actually in the commutator space $\Com(\sL_g)$, and so it is annihilated by \emph{every} trace on $\sL_g$. 
\end{remark}

\subsection{Schr\"odinger operators on unbounded domains}
On $\R^n$, $n\geq 2$, consider the family of Schr\"odinger operators,
\begin{equation*}
 H_\alpha:= - \Delta + |x_1 \cdots x_n|^\alpha, \qquad \alpha>0,
\end{equation*}
where $\Delta=-(\partial_{x_1}^2+\cdots + \partial_{x_n}^2)$ is the (positive) Laplacian.

It was shown by Simon~\cite{Si:AP83} in dimension $n=2$ that $H_\alpha$ is essentially selfadjoint and its spectrum is positive and discrete. This was extended by Camus-Rautenberg~\cite{CR:JMP15} to dimensions~$\geq 3$. It was further shown by Simon~\cite{Si:JFA83} (for $n=2$) and Camus-Rautenberg~\cite{CR:JMP15} (for $n\geq 3$) that as $\lambda\rightarrow \infty$ we have the Weyl law,
\begin{equation}\label{eq:distribution of H_alpha}
 N\big(H_\alpha;\lambda\big) \sim c(n,\alpha) \lambda^{\frac{n}{2}+\frac{1}{\alpha}} (\log \lambda)^{n-1},
\end{equation}
where we have set
\begin{equation}\label{eq:Examples.Simon-cnalpha}
  c(n,\alpha) = \frac{1}{(n-1)!}\bigg(\frac{n}{2}+\frac1\alpha\bigg)^{n-1}
   \frac{\pi^{-\frac{n}{2}}\Gamma\big(\frac1\alpha+1\big)}{\Gamma\big(\frac{n}{2}+\frac1\alpha +1\big)}.
\end{equation}
This type of Weyl law prompted Simon~\cite{Si:JFA83} to coin the terminology ``nonclassical Weyl law''.

By combining~(\ref{eq:distribution of H_alpha}) with Lemma~\ref{lem:example.Ap-spectral-meas} we get the following result.

\begin{proposition}\label{prop:Simon}
 Put $g(t)=(t+1)^{-1}(\log(t+2))^{n-1}$, $t\geq 0$. The operator $H_\alpha^{-(\frac{n}{2}+\frac1\alpha)}$ is spectrally measurable in $\sL_g$, and we have
 \begin{equation*}
 \gbint H_\alpha^{-(\frac{n}{2}+\frac1\alpha)} = \frac{1}{(n-1)!}\bigg(\frac{n}{2}+\frac1\alpha\bigg)^{n-1}
   \frac{\pi^{-\frac{n}{2}} \Gamma\big(\frac1\alpha+1\big)}{\Gamma\big(\frac{n}{2}+\frac1\alpha +1\big)}.
\end{equation*}
\end{proposition}

We note that in the above result the exponent $\alpha^{-1} +n/2$ takes on all positive real values~$>n/2$ as $\alpha$ ranges over $(0,\infty)$. As observed
in~\cite{Si:JFA83,  CR:JMP15} the behavior of the spectrum of $H_\alpha$ as $\alpha \rightarrow \infty$ is closely related to the spectrum of the Dirichlet Laplacian
$\Delta_\Omega$ on the infinite-volume open domain,
\begin{equation*}
 \Omega:=\left\{(x_1,\ldots, x_n)\in \R^n; \ |x_1\cdots x_n|<1\right\}.
\end{equation*}
In particular, we get a Weyl law for $\Delta_\Omega$ by letting $\alpha \rightarrow \infty$ in the r.h.s.\ of~(\ref{eq:distribution of H_alpha}). Namely, as $\lambda \rightarrow \infty$ we have
\begin{equation*}
 N\big(\Delta_\Omega;\lambda\big) \sim c(n,\infty) \lambda^{\frac{n}{2}} (\log \lambda)^{n-1},
\end{equation*}
where
\begin{equation*}
  c(n,\infty) =  \frac{1}{(n-1)!}\bigg(\frac{n}{2}\bigg)^{n-1} \frac{\pi^{-\frac{n}{2}}}{\Gamma\big(\frac{n}{2}+1\big)}=  \frac{2n^n}{n!}(2\pi)^{-n}|\bB^n|.
\end{equation*}
Combining this with Lemma~\ref{lem:example.Ap-spectral-meas} then yields the following result.

\begin{proposition}[compare {\cite[Example~2.7]{GU:JMAA20}}, {\cite[Example~5.5]{TU:JMAA27}}] \label{prop:Examples.Dirichlet}
The operator $\Delta_\Omega^{-n/2}$ is spectrally measurable in $\sL_g$, with $g(t)=(t+1)^{-1}(\log (t+2))^{n-1}$, $t\geq 0$. Moreover, we have
 \begin{equation*}
 \gbint \Delta_\Omega^{-\frac{n}{2}} = \frac{2n^n}{n!}(2\pi)^{-n}|\bB^n|.
 \end{equation*}
\end{proposition}

\begin{remark}
 The above examples were considered in \cite{GU:JMAA20, TU:JMAA27}, but there only Dixmier-measurability was considered. The above result thus provides an improvement by establishing spectral measurability. In particular, this ensures that the operators
 $H_\alpha^{-(\frac{n}{2}+\frac1\alpha)}$  and $\Delta_\Omega^{-n/2}$ are strongly measurable.
\end{remark}

\begin{remark}\label{rmk.Examples.Dirichlet}
 We also stress the contrast between Proposition~\ref{prop:Examples.Dirichlet}  and the corresponding result in the case where $\Omega$ is bounded with, say, a Lipschitz boundary. In the latter case we get spectral measurability (and even $(t+1)^{-1}$-hypermeasurability) in the weak trace class $\sL_{1,\infty}$, since we have the (classical) Weyl law,
 \begin{equation*}
  N\big(\Delta_\Omega;\lambda\big) \sim c(n) \lambda^{\frac{n}{2}}\left( 1+ \op{O}\big(\lambda^{-1/2} \big)\right), \qquad c(n)=(2\pi)^{-n}|\bB^n|\Vol(\Omega).
\end{equation*}
\end{remark}

\subsection{Open manifolds with conformally cusp metrics}
In~\cite{Mo:MA08} S.\ Moroianu established Weyl laws, including nonclassical Weyl laws, for some open manifolds with conformally cusp metrics. More precisely, let
$X^n$ be an open manifold, which is the interior of a compact manifold $\overline{X}$ with closed boundary $M^{n-1}=\partial \overline{X}$ and boundary defining function $x$.  We assume that $X$ is endowed with a Riemannian metric of the form,
\begin{equation}\label{eq:Intro.boundary-metric}
 g=x^{2r}g_0,
\end{equation}
where $g_0$ is a cusp metric on $\overline{X}$ (see~\cite{Mo:MA08}), i.e., a (complete) Riemannian metric which in local coordinates near $M$ takes the form,
\begin{equation*}
 g_0= a_{00}(x,y)x^{-4}dx^2+ \sum_{1\leq j<n}a_{0j}(x,y)x^{-2}dxdy^{j} +  \sum_{1\leq i, j<n}h_{ij}(x,y)dy^i dy^j.
\end{equation*}
If in addition $a_{00|M}=1$ and $a_{0j|M}=0$ for $j=1, \ldots, n$, then (following Melrose~\cite{Me:Book}) we say that $g_0$ is an \emph{exact} cusp metric. Note that the cusp metric $g_0$ induces on $M$ the metric,
\begin{equation}\label{eq:Cusp.boundary-metric}
 h_0:=  \sum_{1\leq i, j<n}h_{ij}(0,y)dy^i dy^j.
\end{equation}
Note that $(X,g)$ has finite volume iff $r>1/n$. It is complete iff $r\geq 1$.

Examples of manifolds with conformally cusp metrics include hyperbolic metrics of finite volume (in which case $r=1$; see~\cite{Mo:MA08}).  Another class of examples is provided by metric horns in the sense of Cheeger~\cite{Ch:PSPM80} (see also~\cite{LP:CCPDE98}) for which $r>1$.

Assume further that $\overline{X}$ is spin. We denote by  $\shD_g$ the Dirac operator of $X$ acting on the sections of the spinor bundle $\shS_X$. In addition, we denote by $\shD_{h_0}$ the Dirac operator of $M=\partial X$ with respect to the metric~(\ref{eq:Intro.boundary-metric}).  We make the following assumptions:
\begin{itemize}
 \item $g_0$ is an exact cusp metric.

 \item $\ker \shD_{h_0}=\{0\}$, i.e., there are no (non-zero) harmonic spinors on $M$.
\end{itemize}
Under these assumptions Moroianu~\cite{Mo:MA08} showed that the Dirac operator $\shD_g$ with domain $C^\infty_c(X)$ is essentially selfadjoint and has purely discrete spectrum. Furthermore, he established Weyl laws for the counting function of $|\shD_g|$. We get classical Weyl laws for $r\neq 1/n$, but for $r=1/n$ we have a nonclassical Weyl law. We will restrict ourselves to the case $r\geq 1/n$, since in those cases the leading coefficients in the Weyl laws are computable locally. As $\lambda \rightarrow \infty$, we have
\begin{equation*}
 N\big(|\shD_g|; \lambda) \sim \left\{
 \begin{array}{cl}
 \slashed{c}_1(n)\Vol_g(X)\lambda^n & \text{if}\ r>\frac{1}{n},\medskip\\
 \slashed{c}_2(n) \Vol_h(M)\lambda^n \log \lambda & \text{if}\ r=\frac{1}{n},
\end{array}\right.
\end{equation*}
where we have set
\begin{equation}\label{eq:Examples.Open-constants}
  \slashed{c}_1(n):= 2^{\left[\frac{n}{2}\right]}(2\pi)^{-\frac{n}{2}} |\bB^{n}|, \qquad
 \slashed{c}_2(n):= 2^{\left[\frac{n}{2}\right]}(2\pi)^{-\frac{n}{2}} |\bS^{n-1}|.
\end{equation}
Combining this with Lemma~\ref{lem:example.Ap-spectral-meas} we then obtain the following result.

\begin{proposition}\label{prop:cusp}
 Suppose that $\overline{X}$ is spin, and assume that $g_0$ is an exact cusp metric and $\ker \shD_{h_0}=\{0\}$.
 \begin{enumerate}
 \item If $r>1/n$, then $|\shD_g|^{-n}$ is spectrally measurable in $\sL_{1,\infty}$, and we have
 \begin{equation*}
 \bint |\shD_g|^{-n} =\slashed{c}_1(n) \Vol_g(X), \qquad \slashed{c}_1(n):= 2^{\left[\frac{n}{2}\right]}(2\pi)^{-\frac{n}{2}} |\bB^{n}|.
\end{equation*}

\item If $r=1/n$, and we put $g(t)=(t+1)^{-1}\log (t+2)$, $t\geq 0$, then $|\shD_g|^{-n}$ is spectrally measurable in $\sL_g$, and we have
 \begin{equation}\label{eq:NCI endpoint case}
 \gbint |\shD_g|^{-n} =\slashed{c}_2(n) \Vol_h(M), \qquad  \slashed{c}_2(n):= 2^{\left[\frac{n}{2}\right]}(2\pi)^{-\frac{n}{2}} |\bS^{n-1}|.
\end{equation}
\end{enumerate}
\end{proposition}

\begin{remark}
 The first part is a special case of the version of Connes' integration formula for open manifolds in~\cite{Po:Open}.
\end{remark}

\begin{remark}
 The formula~(\ref{eq:NCI endpoint case}) shows that, for $r=1/n$, the NC integral recaptures the volume of the boundary $M=\partial X$. As mentioned above, if $r=1/n$, then  $(X,g)$ has infinite volume.
\end{remark}

\subsection{Product of quantum spheres with the 2-torus} We revisit an example of Gayral-Sukochev~\cite[Example~4.9]{GS:JFA14}. Let $q\in (0,1)$. The Podle\'s quantum sphere $\bS^2_q$ is a quantum analogue of the 2-sphere $\bS^2$. Its coordinate algebra $\cO(\bS^2_q)$ appears as a $*$-subalgebra of the coordinate Hopf $*$-algebra $\cO(\SU_q(2))$ which is invariant under the action of $\sU_q(\su(2))$ (see~\cite{Po:LMP87}).

The quantum sphere $\bS^2_q$ is also a quantum space in the sense of Connes' noncommutative geometry, since we have a natural spectral triple $(\cO(\bS^2_q), \sH, \sD_q)$, which was constructed by Dabrowski-Sitarz~\cite{DS:PAN03}. Here $\sH=\sH_{1/2}\oplus \sH_{1/2}$ is two copies of the same Hilbert space $\sH_{1/2}$, which is actually the representation space of two non-isomorphic $\sU_q(\su(2))$-equivariant representations of $\cO(\bS^2_q)$.

The Dirac operator $\sD_q$ is a selfadjoint unbounded operator on $\sH=\sH_{1/2}\oplus \sH_{1/2}$ of the form,
\begin{equation*}
 \sD_q =
 \begin{pmatrix}
0 & D_q\\
 D_q& 0
\end{pmatrix},
\qquad D_q\xi_{\ell,m} = \bigg[\ell+\frac12\bigg]_q\xi_{\ell,m},  \qquad [x]_q = \frac{q^x-q^{-x}}{q-q^{-1}},
 \end{equation*}
 where $(\xi_{\ell,m})$ is some orthonormal basis of $\sH_{1/2}$ parametrized by $\ell \in \N_0+1/2$ and $m\in \{-\ell,-\ell+1, \ldots,\ell\}$. Note that $|\sD_q|=D_q\oplus D_q$. Moreover, it is shown in~\cite{EIS:CMP14} that
 \begin{equation}\label{eq:1st heat kernel formula}
 \Tr\left[e^{-t|\sD_q|}\right] = \bigg(\frac{\log t}{\log q}\bigg)^2\big( 1 +\op{O}\big((\log t)^{-1}\big)\big)  \qquad \text{as} \ t\rightarrow 0^+.
\end{equation}

 Let $\Delta=-(\partial_{x}^2+\partial_{y}^2)$ be the Laplacian on the (ordinary) torus $\T^2=(\R/(2\pi \Z))^2$. This is an essentially selfadjoint operator with non-negative and discrete spectrum. Its eigenvalues are $|k|^2$, $k\in \Z^2$. Moreover, as $t\rightarrow 0^+$ we have
 \begin{equation}\label{eq:2nd heat kernel formula}
 \Tr\left[e^{-t\Delta}\right]= (4\pi t)^{-1}\big(\Vol(\T^2)+ \op{O}(t)\big)= \pi t^{-1}\big(1+ \op{O}(t)\big).
\end{equation}
We actually can get a much better remainder term by using the Poisson summation formula (see, e.g., \cite{Po:arXiv26}). However,  the above remainder term is enough for our purpose.

 Following Gayral-Sukochev~\cite[Example~4.9]{GS:JFA14} we look at the operator,
 \begin{equation*}
A:=|\sD_q|\otimes 1 + 1 \otimes \Delta.
\end{equation*}
 This is an essentially selfadjoint unbounded operator on $\sH\otimes L^2(\T^2)$. Moreover, by using~(\ref{eq:1st heat kernel formula}) and~(\ref{eq:2nd heat kernel formula}) we see that, as $t\rightarrow 0^+$, we have
 \begin{equation*}
 \Tr\left[e^{-tA}\right]= \Tr\left[e^{-t|\sD_q|}\right] \Tr\left[e^{-t\Delta}\right]= \frac{\pi}{(\log q)^2} t^{-1}(\log t)^2\big( 1 +\op{O}\big((\log t)^{-1}\big)\big).
\end{equation*}
Therefore, we may apply Karamata's Tauberian theorem in the form, e.g., of \cite[Theorem IV.8.1]{Ko:Springer04} to get the Weyl law,
\begin{equation*}
\lim_{\lambda \rightarrow \infty} \lambda^{-1} (\log \lambda)^{-2}  N(A;\lambda)= \frac{1}{\Gamma(2)} \lim_{t\rightarrow 0^+} t(\log t)^{-2}  \Tr\left[e^{-tA}\right] =  \frac{\pi}{(\log q)^2}.
\end{equation*}
 By combining this with Lemma~\ref{lem:example.Ap-spectral-meas} we then arrive at the following result.

\begin{proposition}[compare {\cite[Example~4.9]{GS:JFA14}}]\label{prop:Podles}
 Put $g(t)=(t+1)^{-1}(\log(t+2))^2$, $t\geq 0$. The operator $(|\sD_q|\otimes 1 + 1 \otimes \Delta)^{-1}$ is spectrally measurable in $\sL_{g}$, and we have
 \begin{equation*}
 \gbint (|\sD_q|\otimes 1 + 1 \otimes \Delta)^{-1} = \frac{\pi}{(\log q)^2}.
\end{equation*}
\end{proposition}

\begin{remark}\label{rmk:Examples.Podles}
 The above result refines \cite[Example~4.9]{GS:JFA14} in two ways (see also~\cite[Example~2.10]{GU:JMAA20}). First, we obtain measurability in the weak Lorentz ideal $\sL_g$, whereas in~\cite{GS:JFA14} they work with the larger Lorentz ideal $\sM_G$, with $G(t)\sim (1/3)[\log(t+2)]^3$. Second, the authors in~\cite{GS:JFA14} focus on measurability with respect to Dixmier traces, but Proposition~\ref{prop:Podles} yields spectral measurability, which is a much stronger property. In particular, this implies strong measurability with respect to all normalized positive traces on $\sL_g$.
\end{remark}

\subsection{Final comments} We refer to~\cite{DR:LNM87, GU:JMAA20, LU:JMAA25, TU:JMAA27} and the companion paper~\cite{Po-Ti:Part2} for additional instances of nonclassical  Weyl laws. Combining them with Lemma~\ref{lem:example.Ap-spectral-meas} produces further examples of spectrally measurable 
operators in weak Lorentz ideals.

\appendix

\section{Proof of results on weak Lorentz ideals}\label{app:proofs}
In this appendix, for the reader's convenience we gather proofs of several results on weak Lorentz ideals that were mentioned in Section~\ref{sec:Lorentz}.

\begin{proof}[Proof of Lemma~\ref{lem:Lorentz.quasi-norm}]
 We only have to prove part (ii) (\emph{cf}.\ Remark~\ref{rmk:quasi-norm}).  Set
 \begin{equation*}
C=\sup_{k\geq 0}C_k, \quad \text{where}\ C_k:= \left(\max\left\{ \frac{g(k)}{g(2k)},  \frac{g(k)}{g(2k+1)},  \frac{g(k+1)}{g(2k+1)}\right\}\right).
\end{equation*}
It follows from Remark~\ref{rmk:slow decreasing} that $C<\infty$. Let $S,T\in \sK$ and $j\geq 0$. Set $k=\lfloor j/2\rfloor$. By~(\ref{eq:Lorentz.Fan1}), we have
\begin{equation*}
 g(j)^{-1} \mu_j(S+T) \leq \frac{g(k)}{g(j)} g(k)^{-1}\mu_k(S) + \frac{g(j-k)}{g(j)} g(j-k)^{-1}\mu_{j-k}(T).
\end{equation*}
If $j$ is even, then $j=2k$, and so
\begin{equation*}
  \frac{g(k)}{g(j)} = \frac{g(j-k)}{g(j)} = \frac{g(k)}{g(2k)}\leq C_k.
\end{equation*}
If $j$ is odd, then $j=2k+1$, and so we get
\begin{equation*}
  \frac{g(k)}{g(j)}  =  \frac{g(k)}{g(2k+1)}\leq C_k \qquad \text{and} \qquad   \frac{g(j-k)}{g(j)} = \frac{g(k+1)}{g(2k+1)}\leq C_k.
\end{equation*}
In any case, we have
\begin{equation*}
  g(j)^{-1} \mu_j(S+T) \leq C_k \left( g(k)^{-1}\mu_k(S) +g(j-k)^{-1}\mu_{j-k}(T)\right).
\end{equation*}
It then follows that
\begin{equation}\label{eq:quasi-triangle inequality}
 \|S+T\|_g \leq C\left(\|S\|_g + \|T\|_g\right) \qquad \forall S,T\in \sK.
\end{equation}
This proves (ii). The proof of Lemma~\ref{lem:Lorentz.quasi-norm} is complete.
\end{proof}

\begin{remark}
 By using Remark~\ref{rmk:slow decreasing} and the fact that $g$ is $\RV_\rho$, it can be shown that $C_k\rightarrow  2^{|\rho|}$ as $k\rightarrow \infty$.
It then follows that
\begin{equation*}
 C=\sup_{k\geq 0}C_k\geq 2^{|\rho|}.
\end{equation*}
\end{remark}

\begin{proof}[Proof of Proposition~\ref{prop:Lorentz.quasi-Banach}]
 We only need to check that every Cauchy sequence $(T_n)_{n\geq 0}$ in $\sL_g$ is convergent. Note that the inclusion of $\sL_g$ into $\sK$ is continuous, since we have
 \begin{equation*}
 \|T\| = \mu_0(T) \leq g(0) \|T\|_g.
\end{equation*}
Thus, $(T_n)_{n\geq 0}$ is a Cauchy sequence in $\sK$, and hence it converges to some operator $T$ in $\sK$.

As   $(T_n)_{n\geq 0}$ is a Cauchy sequence in $\sL_g$, given any $\epsilon>0$, there is $N\geq 0$ such that, for all $m,n\geq N$ and $j\geq 0$, we have
\begin{equation*}
g(j)^{-1} \mu_j(T_m-T_n) \leq \|T_m-T_n\|_g \leq \epsilon.
\end{equation*}
Letting $m\rightarrow \infty$ and using~(\ref{eq:Quantized.properties-mun2}) we get
\begin{equation*}
 g(j)^{-1} \mu_j(T-T_n) \leq \epsilon \qquad \forall j \geq 0.
\end{equation*}
Thus,
\begin{equation}\label{eq:convergence}
 \|T-T_n\|_g \leq \epsilon \qquad \forall n>N.
\end{equation}
In particular, by using~(\ref{eq:quasi-triangle inequality}) we get
\begin{equation*}
 \|T\|_g \leq C\left( \|T-T_n\|_g + \|T_n\|_g\right) <\infty.
\end{equation*}
That is, $T\in \sL_g$. The estimate~(\ref{eq:convergence}) then means that $T_n\rightarrow T$ in $\sL_g$. This shows that every Cauchy sequence in $\sL_g$ is convergent. The
proof of Proposition~\ref{prop:Lorentz.quasi-Banach} is complete.
\end{proof}

\begin{proof}[Proof of Lemma~\ref{lem:Lorentz.distance-formula}]
Let $T\in \sL_g$ have Schmidt series  $T=\sum_{j \geq 0} \mu_j(T)\ketbra{U\xi_j}{\xi_j}$, and set $T_N=\sum_{j<N} \mu_j(T) \ketbra{U\xi_j}{\xi_j}$, $N\geq 1$. As $T_N$ has finite rank, we have $\|T-T_N\|_g \geq  \dist_{\sL_g}(T;\sR)$, and hence
 \begin{equation*}
  \dist_{\sL_g}(T;\sR) \leq  \limsup_{N\rightarrow \infty} \|T-T_N\|_g.
\end{equation*}

Let $R\in \sR$, and set $N=\rk(R)$. The min-max principle~(\ref{eq:min-max}) implies that $\mu_N(R)=0$. Combining this with Fan's inequality~(\ref{eq:Lorentz.Fan1}) we see that, for $j\geq N$, we have
\begin{equation*}
 \mu_j(T) \leq \mu_{j-N}(T-R)+\mu_N(R) =\mu_{j-N}(T-R).
\end{equation*}
 Thus,
 \begin{equation*}
 g(j)^{-1} \mu_j(T) \leq \frac{g(j-N)}{g(j)} g(j-N)^{-1}\mu_{j-N}(T-R) \leq \frac{g(j-N)}{g(j)}\|T-R\|_g.
\end{equation*}
Using Remark~\ref{rmk:slow decreasing} then gives
\begin{equation*}
  \limsup_{j\rightarrow \infty} g(j)^{-1}\mu_j(T) \leq \|T-R\|_g.
\end{equation*}
As the above inequality holds for every $R\in \sR$, we get
\begin{equation*}
  \limsup_{j\rightarrow \infty} g(j)^{-1}\mu_j(T) \leq  \dist_{\sL_g}(T;\sR) .
\end{equation*}

To complete the proof it is enough to show that
\begin{equation}\label{eq:auxiliary inequality}
 \limsup_{N\rightarrow \infty} \|T-T_N\|_g \leq  \limsup_{j\rightarrow \infty} g(j)^{-1}\mu_j(T).
\end{equation}
To see this, set
\begin{equation*}
 S_N:= \sum_{j \geq N}  \mu_j(T)\ketbra{\xi_j}{\xi_j}= \sum_{j\geq 0}  \mu_{j+N}(T)\ketbra{\xi_{j+N}}{\xi_{j+N}} \qquad N\geq 1.
\end{equation*}
As $S_N\geq 0$ with eigensequence $(\mu_{j+N}(T))_{j\geq 0}$, we see that $\mu_j(S_N)=\mu_{j+N}(T)$ for all $j\geq 0$. Note also that
\begin{equation*}
 T-T_N= \sum_{j \geq N}  \mu_j(T)\ketbra{U\xi_j}{\xi_j}=US_N.
\end{equation*}
Here $U^*U$ is the orthogonal projection onto $\ker T=\ker |T|$. As $\ran S_N \subseteq \ran |T| \subseteq (\ker |T|)^\perp$ we see that $U^*US_N=S_N$. Thus,
\begin{equation*}
 (T-T_N)^*(T-T_N)=S_N U^*US_N=S_N^2.
\end{equation*}
It follows that $|T-T_N|=S_N$, and so we have
\begin{equation*}
 \mu_j(T-T_N)= \mu_j(S_N)=\mu_{j+N}(T).
\end{equation*}
Thus,
\begin{equation*}
 g(j)^{-1}\mu_j(T-T_N) =  \frac{g(j+N)}{g(j)}g(j+N)^{-1}\mu_{j+N}(T) \leq \rho(N) g(j+N)^{-1}\mu_{j+N}(T),
\end{equation*}
where we have set
\begin{equation*}
 \rho(N):= \sup_{j\geq 0} \frac{g(j+N)}{g(j)}.
\end{equation*}
It then follows that, for all $N\geq 1$, we have
\begin{equation}\label{eq:domi inequality}
 \|T-T_N\|_g = \sup_{j\geq 0}  g(j)^{-1}\mu_j(T-T_N) \leq \rho(N) \sup_{j\geq N} g(j)^{-1}\mu_j(T).
\end{equation}

By assumption $g(t)$ is decreasing on $[t_0,\infty)$. We then have $\rho(N)=\max\{\rho_1(N), \rho_2(N)\}$, where
\begin{equation*}
 \rho_1(N) = \max_{0\leq j \leq t_0} \frac{g(j+N)}{g(j)}, \qquad  \rho_2(N) = \sup_{ j \geq t_0} \frac{g(j+N)}{g(j)}\leq 1.
\end{equation*}
As $g(t)\rightarrow 0$, we see that $ \rho_1(N)\rightarrow 0$ as $N\rightarrow \infty$. It follows that $\rho(N)\leq 1$ for $N$ large enough. Combining this with~(\ref{eq:domi inequality}) we then get
\begin{equation*}
\limsup_{N\rightarrow \infty} \|T-T_N\|_g \leq \lim_{N\rightarrow \infty} \sup_{j\geq N} g(j)^{-1}\mu_j(T)=  \limsup_{j\rightarrow \infty} g(j)^{-1}\mu_j(T).
\end{equation*}
This proves~(\ref{eq:auxiliary inequality}).  The proof of Lemma~\ref{lem:Lorentz.distance-formula} is complete.
\end{proof}

\begin{proof}[Proof of Proposition~\ref{prop:Lorentz.Marcinkiewicz}]
 Let $T\in \sK$. As $\mu_j(T)\leq \|T\|_g g(j)$, we have
 \begin{equation*}
 \frac{1}{G(N)} \sum_{j<N} \mu_j(T) \leq \frac{\|T\|_g}{G(N)} \sum_{j\leq N} g(j).
\end{equation*}
By assumption $g$ is a positive continuous function that is decreasing on $[t_0,\infty)$. If $\int_0^\infty g(t)dt=\infty$, then $\sum_{j\leq N}g(j)\simeq G(N)$ as $N\rightarrow \infty$. If $\int_0^\infty g(t)dt<\infty$, then $\sum_{j\geq 0} g(j)<\infty$. In either case the sequence $\frac{1}{G(N)} \sum_{j<N} g(j)$ is bounded. Therefore, we have
\begin{equation*}
 \|T\|_{G} \leq C \|T\|_g, \qquad \textup{with}\ C:= \sup_{N\geq 1} \frac{1}{G(N)} \sum_{j<N} g(j) <\infty.
\end{equation*}
This gives the first part.

Assume $-1<\rho<0$. The inequality $N\mu_N(T) \leq \sum_{j<N}\mu_j(T)$ yields
\begin{equation*}
 g(N)^{-1} \mu_N(T) \leq \frac{G(N)}{Ng(N)} \frac{1}{G(N)} \sum_{j<N} \mu_j(T) \leq \frac{G(N)}{Ng(N)}  \|T\|_{G}.
\end{equation*}
As  $-1<\rho<0$, it follows from Karamata's theorem (Theorem~\ref{thm:karamata}) that $\frac{G(t)}{tg(t)}$ is bounded on $[0,\infty)$. Thus,
\begin{equation*}
 \|T\|_g \leq C' \|T\|_{G}, \qquad \textup{where}\ C':= \sup_{N\geq 1} \frac{G(N)}{Ng(N)}<\infty.
\end{equation*}
This gives the second part. The proof of Proposition~\ref{prop:Lorentz.Marcinkiewicz} is complete.
 \end{proof}

\section{Counting Functions and Eigenvalues} \label{app:counting}
In this appendix, we link the asymptotics of non-increasing non-negative null sequences to the asymptotic behaviors of their counting functions when the leading terms are given by functions of regular variation and suitable asymptotic inverses.

\subsection{Asymptotic inverses}
If $h:[0,\infty)\rightarrow [0,\infty)$ is a function converging to $\infty$ as $t\rightarrow \infty$, then by an \emph{asymptotic inverse} we shall mean any function
$h^\sharp:[0,\infty)\rightarrow [0,\infty)$ converging to $\infty$ as $t\rightarrow \infty$ such that
 \begin{equation*}
 h(h^\sharp(t)) \sim t \qquad \text{and} \qquad h^\sharp(h(t)) \sim t  \qquad \text{as}\ t\rightarrow \infty.
 \end{equation*}

As the following proposition shows, every $\RV$-function of positive index admits an asymptotic inverse which is an $\RV$-function as well.

\begin{proposition}[see {\cite[Theorem~5.1.12]{BGT:Cambridge87}}] \label{prop:counting-asymptotic-inverse}
Let $h:[0,\infty)\rightarrow [0,\infty)$ be an $\RV_p$-function with $p>0$.
\begin{enumerate}
 \item $h(t)$ always admits an asymptotic inverse $h^\sharp(t)$ which is $\RV_{1/p}$.

 \item A function $\tilde{h}(t)$ is an asymptotic inverse of $h(t)$ if and only if $\tilde{h}(t)\sim h^\sharp(t)$.
\end{enumerate}
 \end{proposition}

 \begin{remark}\label{rmk:B2}
  It follows from Remark~\ref{rmk:continuous rv} and Proposition~\ref{prop:counting-asymptotic-inverse}
 that we can always choose the asymptotic inverse to be a continuous $\RV_{1/p}$-function.
If $h:[0,\infty)\rightarrow [0,\infty)$ is an increasing bijection, then $h$ and $h^{-1}$ are continuous and are asymptotic inverses of each other.
 In general, if $h$ is only ultimately increasing, then there is $a>0$ such that $h:[a,\infty)\rightarrow [h(a),\infty)$ is an increasing bijection. An instance of a continuous asymptotic inverse of $h$ is then given by any continuous function $h^\natural:[0,\infty)\rightarrow [0,\infty)$ such that $h^\natural=(h_{|[a,\infty)})^{-1}$ on $[h(a),\infty)$.
 \end{remark}

\begin{example}
 If $h(t)=t^p$, $p>0$, then its inverse $h^\natural(t)=t^{1/p}$, $t\geq 0$, is an $\RV_{1/p}$-continuous asymptotic inverse.
\end{example}

 \begin{example}\label{ex:asymp inverse}
 If $h(t)=t^p(\log (t+2))^q$, $p>0$, $q\in \R$, then an $\RV_{1/p}$-continuous asymptotic inverse is given by
 \begin{equation*}
 h^\sharp(t):=p^{\frac{q}{p}} t^{\frac1p} \left(\log (t+2)\right)^{-\frac{q}{p}}, \qquad t\geq 0.
\end{equation*}
\end{example}

\subsection{Asymptotic behavior of counting functions}
From now on, we let
\begin{equation*}
 \lambda_0\geq \lambda_1 \geq \lambda_2 \geq \cdots \geq 0
\end{equation*}
be a non-increasing non-negative sequence converging to $0$. Its counting function is given by
\begin{equation*}
 N(\lambda):= \#\{j; \lambda_j>\lambda\}, \qquad \lambda>0.
\end{equation*}

In addition, we let $h:[0,\infty)\rightarrow [0,\infty)$ be an $\RV_p$ function with $p>0$, and let $h^\sharp: [0,\infty)\rightarrow [0,\infty)$ be an
$\RV_{1/p}$ asymptotic inverse.

\begin{lemma}\label{lem:B5}
We have
 \begin{gather} \label{eq:limsup in appendix B}
\limsup_{j\rightarrow \infty} h^\sharp(j)\lambda_j=\bigg[\limsup_{\lambda\rightarrow 0^+} h(\lambda^{-1})^{-1}N(\lambda)\bigg]^{\frac1p},\\
\liminf_{j\rightarrow \infty} h^\sharp(j)\lambda_j=\bigg[\liminf_{\lambda\rightarrow 0^+} h(\lambda^{-1})^{-1}N(\lambda)\bigg]^{\frac1p}
\label{eq:liminf in appendix B}.
\end{gather}
 Thus,
\begin{equation}\label{eq:asymptotic equivalence}
 \lim_{\lambda\rightarrow 0^+} h(\lambda^{-1})^{-1}N(\lambda)=c \ \Longleftrightarrow \ \lim_{j\rightarrow \infty} h^\sharp(j)\lambda_j =c^{\frac1p}.
\end{equation}
 \end{lemma}
 \begin{proof}
 We only have to prove~(\ref{eq:limsup in appendix B})--(\ref{eq:liminf in appendix B}). Define
\begin{gather*}
 \underline{N}=\liminf_{\lambda\rightarrow 0^+} h\left(\lambda^{-1}\right)^{-1}N(\lambda), \qquad \overline{N}=\limsup_{\lambda\rightarrow 0^+}
 h\left(\lambda^{-1}\right)^{-1}N(\lambda),\\
 \underline{L}=\liminf_{j\rightarrow \infty} h^\sharp(j)\lambda_j, \qquad  \overline{L}=\limsup_{j\rightarrow \infty} h^\sharp(j)\lambda_j.
\end{gather*}
To complete the proof we then only need to show the following two sets of inequalities:
\begin{gather}
 \underline{N}^{\frac1p} \leq \underline{L} \qquad \textup{and} \qquad  \overline{L}  \leq \overline{N}^{\frac1p},\label{eq:NLLN}\\
  \underline{L}^p\leq \underline{N} \qquad \textup{and} \qquad  \overline{N} \leq \overline{L}^p. \label{eq:LNNL}
\end{gather}

\noindent \texttt{Proof of~(\ref{eq:NLLN}).} We may assume $\underline{N}>0$. Let $a\in(0,1)$. The definition of $N(\lambda)$ implies that
\begin{equation*}
N(\lambda_j)\leq j\leq N(a\lambda_j)-1\leq N(a\lambda_j).
\end{equation*}
In addition, given any $\epsilon>0$, the definitions of $\underline{N}$ and $\overline{N}$ imply that, for $\lambda>0$ small enough, we have
\begin{equation*}
(1-\epsilon)\underline{N}\leq h(\lambda^{-1})^{-1}N(\lambda)\leq \overline{N}+\epsilon.
\end{equation*}
Take $\lambda=\lambda_j$ and $\lambda=a\lambda_j$, respectively. Then for $j$ large enough, we have
\begin{equation*}
	(1-\epsilon)\underline{N}h(\lambda_j^{-1})\leq N(\lambda_j)\leq j\leq N(a\lambda_j)\leq (\overline{N}+\epsilon)h(a^{-1}\lambda_j^{-1}).
\end{equation*}
As $h^\sharp$ is ultimately increasing, for $j$ large enough, we get
\begin{equation*}
	h^\sharp[(1-\epsilon)\underline{N}h(\lambda_j^{-1})]\leq h^\sharp(j)\leq h^\sharp[(\overline{N}+\epsilon)h(a^{-1}\lambda_j^{-1})].
\end{equation*}
Thus,
\begin{equation*}
	h^\sharp[(1-\epsilon)\underline{N}h(\lambda_j^{-1})]\lambda_j\leq h^\sharp(j)\lambda_j\leq h^\sharp[(\overline{N}+\epsilon)h(a^{-1}\lambda_j^{-1})]\lambda_j.
\end{equation*}
As $h^\sharp$ is $\RV_{1/p}$ and is an asymptotic inverse of $h$, for $j\to\infty$ we have
\begin{equation*}
 h^\sharp[(\overline{N}+\epsilon)h(a^{-1}\lambda_j^{-1})]\lambda_j\sim  (\overline{N}+\epsilon)^{\frac1p}h^\sharp[h(a^{-1}\lambda_j^{-1})]\lambda_j\sim a^{-1}(\overline{N}+\epsilon)^{\frac1p}.
\end{equation*}
Similarly, we have
\begin{equation*}
	h^\sharp[(1-\epsilon)\underline{N}h(\lambda_j^{-1})]\lambda_j\sim [(1-\epsilon)\underline{N}]^{\frac1p}	h^\sharp[h(\lambda_j^{-1})]\lambda_j\sim [(1-\epsilon)\underline{N}]^{\frac1p}.
\end{equation*}
We then get
\begin{equation*}
[(1-\epsilon)\underline{N}]^{\frac1p}\leq \underline{L}\leq \overline{L}\leq a^{-1}(\overline{N}+\epsilon)^{\frac1p}.
\end{equation*}
Letting $\epsilon\to 0^+$ and $a\to1^{-}$ gives the inequalities ~(\ref{eq:NLLN}).

\noindent \texttt{Proof of ~(\ref{eq:LNNL}).} We may assume that $\underline{L}>0$. The proof relies on the inequalities:
\begin{equation*}
\lambda_{N(\lambda)}\leq \lambda\leq	\lambda_{N(\lambda)-1}.
\end{equation*}
Let $\epsilon>0$. For $j$ large enough, we have
\begin{equation*}
(1-\epsilon)\underline{L}\leq h^\sharp(j)\lambda_j\leq \overline{L}+\epsilon.
\end{equation*}
Then, for $j=N(\lambda)$ and $j=N(\lambda)-1$, we get
\begin{equation*}
(1-\epsilon)\underline{L} h^\sharp(N(\lambda))^{-1}\leq\lambda_{N(\lambda)}\leq \lambda\leq	\lambda_{N(\lambda)-1}\leq(\overline{L}+\epsilon)h^\sharp(N(\lambda)-1)^{-1}.
\end{equation*}
As the function $\lambda\to h(\lambda^{-1})^{-1}$ is increasing near $0$, for $\lambda>0$ small enough, we have
\begin{equation*}
h[((1-\epsilon)\underline{L})^{-1} h^\sharp(N(\lambda))]^{-1}\leq h(\lambda^{-1})^{-1}\leq h[(\overline{L}+\epsilon)^{-1}h^\sharp(N(\lambda)-1)]^{-1}.
\end{equation*}
Thus,
\begin{equation}\label{eq:h-full}
h[((1-\epsilon)\underline{L})^{-1} h^\sharp(N(\lambda))]^{-1}N(\lambda)\leq h(\lambda^{-1})^{-1}N(\lambda)\leq h[(\overline{L}+\epsilon)^{-1}h^\sharp(N(\lambda)-1)]^{-1}N(\lambda).
\end{equation}
As $h$ is an $\RV_p$-function and is an asymptotic inverse of $h^\sharp$, as $\lambda\to0^+$, we have
\begin{equation}\label{eq:h-left}
h[((1-\epsilon)\underline{L})^{-1} h^\sharp(N(\lambda))]^{-1}N(\lambda)\sim [(1-\epsilon)\underline{L}]^p[h\circ h^\sharp(N(\lambda))]^{-1}N(\lambda)\sim [(1-\epsilon)\underline{L}]^p.
\end{equation}
Similarly, we have
\begin{equation*}
h[(\overline{L}+\epsilon)^{-1}h^\sharp(N(\lambda)-1)]^{-1}N(\lambda)\sim (\overline{L}+\epsilon)^p[h\circ h^\sharp(N(\lambda)-1)]^{-1}N(\lambda)\sim (\overline{L}+\epsilon)^p.
\end{equation*}
Combining this with~(\ref{eq:h-full})--(\ref{eq:h-left}), we then get
\begin{equation*}
[(1-\epsilon)\underline{L}]^p\leq \underline{N} \quad\mbox{and}\quad \overline{N}\leq (\overline{L}+\epsilon)^p.
\end{equation*}
Letting $\epsilon\to0^+$ gives~(\ref{eq:LNNL}). The proof is complete.
\end{proof}

\begin{example}
 Suppose that $h(t)=t^p$, $p>0$. An exact inverse of $h(t)$ then is
\begin{equation*}
 h^\natural(t):=t^{\frac1p}, \qquad t\geq 0.
\end{equation*}
This is an $\RV_{\frac1p}$-continuous function. Therefore, Lemma~\ref{lem:B5} allows us to recover the well-known equivalence,
 \begin{equation*}
 \lim_{\lambda\rightarrow 0^+} \lambda^pN(\lambda)=c \ \Longleftrightarrow \ \lim_{j\rightarrow \infty} j^{\frac1p}\lambda_j =c^{\frac1p}.
\end{equation*}
\end{example}

\begin{example}
 Suppose that $h(t)=t^p(\log(t+2))^q$ with $p>0$ and $q\neq 0$. As mentioned in Example~\ref{ex:asymp inverse}, an $\RV_{1/p}$ asymptotic inverse is given by
\begin{equation*}
 h^\sharp(t):= p^{\frac{q}{p}} t^{\frac{1}{p}} \left(\log(t+2)\right)^{-\frac{q}{p}}, \qquad t\geq 0.
\end{equation*}
We may thus apply Lemma~\ref{lem:B5} to get the following equivalence, 
\begin{equation*}
 \lim_{\lambda\rightarrow 0^+} \lambda^p|\log \lambda|^{-q}N(\lambda)=c \ \Longleftrightarrow \ \lim_{j\rightarrow \infty} j^{\frac1p} (\log j)^{-\frac{q}{p}}\lambda_j =
 \left(p^{-q}c\right)^{\frac1p}.
\end{equation*}
\end{example}

\end{document}